\documentclass[twoside, 12pt]{article}
\usepackage{a4wide,amsmath, amsthm, amssymb, amscd, mathptmx}
\usepackage{amsfonts}
\usepackage[all]{xy}\CompileMatrices\SelectTips{cm}{12}

\theoremstyle{plain}
\newtheorem{Thm}{\sc Theorem}[section]
\newtheorem{Theorem}[Thm]{\sc Theorem}
\newtheorem{Corollary}[Thm]{\sc Corollary}
\newtheorem{Proposition}[Thm]{\sc Proposition}
\newtheorem{Lemma}[Thm]{\sc Lemma}

\theoremstyle{definition}
\newtheorem{Definition}[Thm]{Definition}

\theoremstyle{remark}
\newtheorem{Remark}[Thm]{Remark}

\newtheorem{Example}[Thm]{Example}
\newtheorem*{Example*}{Example}
\newtheorem*{Remark*}{Remark}

\newcommand{\cA}{{\mathcal A}}

\newcommand{\cD}{{\mathcal D}}

\newcommand{\cJ}{{\mathcal J}}

\newcommand{\cM}{{\mathcal M}}

\newcommand{\cO}{{\mathcal O}}

\newcommand{\cS}{{\mathcal S}}

\newcommand{\cU}{{\mathcal U}}

\newcommand{\cX}{{\mathcal X}}

\renewcommand{\AA}{{\mathbb A}}

\newcommand{\FF}{{\mathbb F}}
\newcommand{\GG}{{\mathbb G}}
\newcommand{\HH}{{\mathbb H}}

\newcommand{\PP}{{\mathbb P}}
\newcommand{\QQ}{{\mathbb Q}}

\newcommand{\VV}{{\mathbb V}}
\newcommand{\WW}{{\mathbb W}}
\newcommand{\ZZ}{{\mathbb Z}}

\newcommand{\uM}{{\underline M}}

\newcommand{\ad}{\mathop{\rm ad}}
\newcommand{\can}{\mathop{\rm can}}
\newcommand{\Coh}{\mathop{\rm Coh}}
\newcommand{\Der}{\mathop{\rm Der}}
\newcommand{\cDer}{\mathop{\mathcal Der}}

\newcommand{\cEnd}{{\mathop{\mathcal{E}nd}\,}}
\newcommand{\Gm}{\mathop{{\mathbb G}_m}}
\newcommand{\Gr}{\mathop{\rm Gr}}

\newcommand{\ti}{\tilde}
\newcommand{\Tr}{\mathop{\rm Tr}}

\newcommand{\id}{{\mathop{\rm id}}}

\newcommand{\Spec}{\mathop{\rm Spec \, }}

\newcommand{\Sch}{{\mathop{{\rm Sch }}}}

\newcommand{\Hom}{{\mathop{{\rm Hom}}}}
\newcommand{\cHom}{{\mathop{{\cal H}om}}}

\renewcommand{\mod}{\mathop{\rm mod}}

\newcommand{\Nilp}{{\mathop{{\mathcal N}ilp }}}

\newcommand{\rk}{\mathop{{\rm rk}}}

\newcommand{\MIC}{{\mathop{{\rm MIC}}}}
\newcommand{\HIG}{{\mathop{{\rm HIG}}}}

\begin{document}

\title{Semistable modules over Lie algebroids in positive characteristic}
\author{Adrian Langer}
\date{March 26, 2014}
\maketitle


{{\sc Address:}\\
Institute of Mathematics, University of Warsaw, ul.\ Banacha 2,
02-097 Warszawa, Poland\\}

\begin{abstract}
We study Lie algebroids in positive characteristic and moduli
spaces of their modules. In particular, we show a Langton's type
theorem for the corresponding moduli spaces. We relate Langton's
construction to Simpson's construction of gr-semistable  Griffiths
transverse filtration. We use it to prove a recent conjecture of
Lan-Sheng-Zuo that semistable systems of Hodge sheaves on liftable
varieties in positive characteristic are strongly semistable.

\end{abstract}

\let\thefootnote\relax\footnote{
Author's work was partially supported by Polish National Science
Centre (NCN) contract number 2012/07/B/ST1/03343.}

\section*{Introduction}

In this paper we give a general approach to relative moduli spaces
of modules over Lie algebroids. As a special case one recovers
Simpson's ``non-abelian Hodge filtration'' moduli space (see
\cite{Si5} and \cite{Si6}). This allows to consider Higgs sheaves
and sheaves with integrable connections at the same time as
objects corresponding to different fibers of the relative moduli
space of modules over a deformation of a Lie algebroid over an
affine line.

A large part of the paper is devoted to generalizing various facts
concerning vector bundles with connections to modules over Lie
algebroids. In particular, we introduce restricted Lie algebroids,
which generalize Ekedahl's 1-foliations \cite{Ek}. In positive
characteristic we define a $p$-curvature for modules over
restricted Lie algebroids. This leads to a deformation of the
morphism given by $p$-curvature on the moduli space of modules to
the Hitchin morphism corresponding to the trivial Lie algebroid
structure. In the special case of bundles with connections on
curves this deformation was already studied by Y. Laszlo and Ch.
Pauly \cite{LP}.

We prove Langton's type theorem for the moduli spaces of modules
over Lie algebroids. We compare it via Rees' construction with
Simpson's inductive construction of gr-semistable Griffiths
transverse filtration (see \cite{Si6}), concluding that the latter
must finish.

This leads to the main application of our results. Namely, we
obtain a canonical gr-semistable Griffiths transverse filtration
on a module over a Lie algebroid. This implies a recent conjecture
of Lan-Sheng-Zuo that semistable systems of Hodge sheaves on
liftable varieties in positive characteristic are strongly
semistable.

The rank 2 case of this conjecture was proven in \cite{LSZ}, the
rank 3 case in \cite{Li}. Recently, independently of the author
Lan, Sheng, Yang and Zuo \cite{LSYZ} also proved the Lan-Sheng-Zuo
conjecture using a similar approach. However, they give a
different proof that Simpson's inductive construction must finish.
They also obtain a slightly weaker result proving their conjecture
only for an algebraic closure of a finite field.

\medskip
The results of this paper are used in \cite{La4} to prove
Bogomolov's type inequality for Higgs sheaves on varieties
liftable modulo $p^2$.

\medskip

\subsection{Notation}

If $X$ is a scheme and $E$ is a quasi-coherent $\cO_X$-module then
we set $E^*=\cHom _{\cO_X}(E, \cO_X)$ and $\VV (E)=\Spec
(S^{\bullet }E)$.

Let $S$ be a scheme of characteristic $p$ (i.e., $\cO_S$ is an
$\FF _p$-algebra). By $F_S^r :S\to S$ we denote the $r$-th
\emph{absolute Frobenius morphism} of $S$ which corresponds to the
$p^r$-th power mapping on $\cO _S$. If $X$ is an $S$-scheme, we
denote  by $X^{(1/S)}$ the fiber product of $X$ and $S$ over the
($1$-st) absolute Frobenius morphism of $S$. The absolute
Frobenius morphism of $X$ induces the \emph{relative Frobenius
morphism} $F_{X/S}: X\to X^{(1/S)}$.

\medskip

Let $X$ be a projective scheme over some algebraically closed field $k$.
Let $\cO _X(1)$ be an ample line bundle on $X$. For any coherent sheaf $E$
on $X$ we define its \emph{Hilbert polynomial} by $P(E)(n)=\chi (X, E(n))$
for $n\in \ZZ$. If $d$ is the dimension of the support of $E$ then we can
write
$$P(E)(n)=\frac{r(E)n^d}{d!}+\hbox{lower order terms in }n.$$
The (rational) number $r=r(E)$ is called the \emph{generalized rank}
of $E$ (note that if $X$ is not integral then the generalized rank of a sheaf depends
on the polarization). The quotient $p(E)=\frac{P(E)}{r(E)}$ is called the
\emph{normalized Hilbert polynomial} of $E$.

In case $X$ is a variety then for a torsion free sheaf $E$ the
generalized rank $r(E)$ is a product of the degree of $X$ with
respect to $\cO_X (1)$ and of the usual rank.

If $X$ is normal and $E$ is a rank $r$ torsion free sheaf on $X$
then we define the \emph{slope} $\mu (E)$ of $E$ as the quotient
of the degree of $\det E=(\bigwedge ^r E)^{**}$ with respect to
$\cO_X(1)$ by the rank $r$. In some cases we consider generalized
slopes defined with respect to a fixed $1$-cycle class, coming
from a collection of nef divisors on $X$.

Let us recall that $E$ is \emph{slope semistable} if for every
subsheaf $E'\subset E$ we have $\mu (E')\le \mu (E)$.

\section{Moduli spaces of modules over sheaves of rings
of differential operators}

In this section we recall some definitions and the theorem on
existence of moduli spaces of modules over sheaves of rings of
differential operators. This combines the results of Simpson
\cite{Si3} with the results of \cite{La2} and \cite{La3}.

Let $S$ be a locally noetherian scheme and let $f: X\to S$ be a
scheme of finite type over $S$. A \emph{sheaf of (associative and unital)
$\cO_S$-algebras $\cA$ on $X$} is a sheaf $\cA$ on $X$ of
(possibly non-commutative) rings of $\cO_X$-bimodules such that
the image of $f^{-1}\cO_S\to \cA$ is contained in the center of
$\cA$.

Let us recall after \cite{Si3} that a \emph{sheaf of rings of
differential operators on $X$ over $S$} is a sheaf $\Lambda$ of
$\cO_S$-algebras on $X$, with a filtration
$\Lambda_0\subset \Lambda_1\subset ...$ by subsheaves of abelian
subgroups satisfying the following properties:
\begin{enumerate}
\item $\Lambda=\bigcup _{i=0}^{\infty}\Lambda_i$ and $\Lambda_i \cdot \Lambda_j\subset
\Lambda_{i+j}$,
\item the image of $\cO _X\to \Lambda$ is equal to $\Lambda_0$,
\item the left and right  $\cO_X$-module structures on $\Gr _i (\Lambda):=\Lambda_i/{\Lambda_{i-1}}$
coincide and the $\cO _X$-modules $\Gr _i (\Lambda)$ are coherent,
\item the sheaf of graded $\cO_X$-algebras $\Gr (\Lambda):=\bigoplus
_{i=0}^{\infty}\Gr _i (\Lambda)$ is generated in degree $1$, i.e.,
the canonical graded morphism from the tensor $\cO_X$-algebra
$T^{\bullet}\Gr _1 (\Lambda )$ of $\Gr _1 (\Lambda )$ to  $\Gr
(\Lambda)$ is surjective.
\end{enumerate}
\medskip

Note that in positive characteristic, the sheaf of rings of
crystalline differential operators (see \cite{BMR} or Subsection
\ref{univ-env-Lie-algebroid-construction}) is a sheaf of rings of
differential operators, but the sheaf of rings of usual
differential operators is not as it almost never is generated in
degree $1$.

\medskip

Assume that $S$ is a scheme of finite type over a universally
Japanese ring $R$. Let $f: X\to S$ be a projective morphism of
$R$-schemes of finite type with geometrically connected fibers and
let $\cO_X(1)$ be an $f$-very ample line bundle. Let $\Lambda$ be
a sheaf of rings of differential operators on $X$ over $S$.

A \emph{$\Lambda$-module} is a sheaf of (left) $\Lambda$-modules
on $X$ which is quasi-coherent with respect to the  induced
$\cO_X$-module structure.

Let $T\to S$ be a morphism of $R$-schemes with $T$ locally
noetherian over $S$. Let us set $X_T=X\times _ST$ and let $p$ be
the projection of $X_T$ onto $X$. Then $\Lambda_T=\cO
_{X_T}\otimes _{p^{-1}\cO_X}p^{-1}\Lambda$ has a natural structure
of a sheaf of rings of differential operators on $X_T$ over $T$.
Moreover, if $E$ is a $\Lambda$-module on $X$ then the pull back
$E_T=p^*E$ has a natural structure of a $\Lambda_T$-module.

Note that if $E$ is a $\Lambda$-module and $E'\subset E$ is a
quasi-coherent $\cO_X$-submodule such that $\Lambda_1\cdot
E'\subset E'$ then $E'$ has a unique structure of $\Lambda$-module
compatible with the $\Lambda$-module structure on $E$ (i.e., such
that $E'$ is a $\Lambda$-submodule of $E$).

Let $Y$ be a projective scheme over an algebraically closed field
$k$ (with fixed polarization) and let $\Lambda _Y$ be a sheaf of
rings of differential operators on $Y$. Let $E$ be a
$\Lambda_Y$-module which is coherent as an $\cO_Y$-module. $E$ is
called \emph{Gieseker (semi)stable} if it is of pure dimension as
an $\cO _Y$-module (i.e., all its associated points have the same
dimension) and for any $\Lambda_Y$-submodule $F\subset E$ we have
inequality $p(F) < p(E)$ ($p(F) \le p(E)$, respectively) of
normalized Hilbert polynomials.

Every Gieseker semistable $\Lambda_Y$-module $E$ has a filtration
$0=E_0\subset E_1\subset ... \subset E_m=E$ by
$\Lambda_Y$-submodules such that the associated graded $\oplus
_{i=0}^m{E_i/E_{i-1}}$ is a \emph{Gieseker polystable} $\Lambda_Y$-module
(i.e., it is a direct sum of Gieseker stable $\Lambda_Y$-modules
with the same normalized Hilbert polynomial). Such a filtration is
called a \emph{Jordan--H\"older filtration} of this
$\Lambda_Y$-module.

\medskip

Now let us go back to the relative situation, i.e., $\Lambda$ on $X$
over $S$ (over $R$).

A \emph{family of Gieseker semistable $\Lambda$-modules on the
fibres of} $p_T:X_T=X\times_S T\to T$ is a $\Lambda_T$-module $E$
on $X_T$ which is $T$-flat (as an $\cO _{X_T}$-module) and such
that for every geometric point $t$ of $T$ the restriction of $E$
to the fibre $X_t$ is pure and Gieseker semistable as a
$\Lambda_t$-module.

We introduce an equivalence relation $\sim $ on such families by
saying that $E\sim E'$ if and only if there exists an invertible
$\cO_T$-module $L$ such that {$E'\simeq E\otimes p_T^* L$.}

Let us define the moduli functor
$$\uM ^{\Lambda}(X/S, P) : (\hbox{\rm Sch/}S) ^{o}\to \hbox{Sets} $$
from the category of locally noetherian schemes over $S$ to the
category of sets by
$$\uM ^{\Lambda}(X/S, P) (T)=\left\{
\aligned
&\sim\hbox{equivalence classes of families of Gieseker}\\
&\hbox{semistable $\Lambda$-modules on the fibres of } X_T\to T,\\
&\hbox{which have Hilbert polynomial }P.\\
\endaligned
\right\} .$$

Then we have the following theorem summing up the results of
Simpson and the author (see \cite[Theorem 4.7]{Si3},
\cite[Theorem~0.2]{La2} and \cite[Theorem 4.1]{La3}).

\begin{Theorem}
Let us fix a polynomial $P$. Then there exists a quasi-projective
$S$-scheme $M ^{\Lambda}(X/S, P)$ of finite type over $S$ and a
natural transformation of functors
$$\varphi :\uM ^{\Lambda}(X/S, P)\to \Hom _S (\cdot, M ^{\Lambda}(X/S, P)),$$
which uniformly corepresents the functor $\uM ^{\Lambda}(X/S, P)$.

For every geometric point $s\in S$ the induced map $\varphi (s)$
is a bijection. Moreover, there is an open scheme $M ^{\Lambda,
s}(X/S, P)\subset M ^{\Lambda}(X/S, P)$ that universally
corepresents the subfunctor of families of geometrically Gieseker
stable $\Lambda$-modules.
\end{Theorem}

In general, for every locally noetherian $S$-scheme $T$ we have a
well defined morphism $M ^{\Lambda}(X/S, P)\times _S T\to M
^{\Lambda _T} (X_T/T, P)$ which is a bijection of sets if $T$ is a
geometric point of $S$.

Let us recall that a scheme $M ^{\Lambda}(X/S, P)$ {\it uniformly
corepresents} $\uM ^{\Lambda}(X/S, P)$ if for every flat base
change $T\to S$ the fiber product $ M ^{\Lambda}(X/S, P) \times _S
T$ corepresents the fiber product functor $\Hom_S (\cdot ,
T)\times _{\Hom _S(\cdot ,S)} \uM ^{\Lambda}(X/S, P)$.

\section{Lie algebroids}

\subsection{Lie algebroids and de Rham complexes}

Let $f:X\to S$ be a morphism of schemes. A \emph{sheaf of
$\cO_S$-Lie algebras on $X$} is a pair $(L, [\cdot, \cdot ]_L)$
consisting of a (left) $\cO_X$-module $L$ (which is an $f^{-1}\cO_S$-bimodule) 
with a morphism of $f^{-1}\cO_S$-modules $[\cdot, \cdot ]_L : L\otimes _{f^{-1}\cO
_S} L\to L$, which is alternating and which satisfies the Jacobi
identity. A homomorphism of sheaves of $\cO_S$-Lie algebras on $X$
is an $\cO_X$-linear morphism $L\to L'$ which preserves the Lie
bracket. As usual for $x\in L(U)$ we define $\ad x: L(U)\to L(U)$
by $(\ad x) (y)=[x,y]_L$.

Let $T_{\cO_S}(L)=\bigoplus_{n\ge 0} \overbrace{L\otimes
_{f^{-1}\cO_S} ...\otimes _{f^{-1}\cO_S}L}^{n}$ be the tensor
algebra of $L$ over $f^{-1}\cO_S$ (it is a non-commutative
$f^{-1}\cO_S$-algebra). Let us recall that the \emph{universal
enveloping algebra} $\cU _{\cO_S}(L)$ of a Lie algebra sheaf $(L,
[\cdot, \cdot ]_L)$ is defined as the quotient of $T_{\cO_S}(L)$
by the two-sided ideal generated by $x\otimes y -y \otimes x
-[x,y]_L$ for all local sections $x, y \in L$.

The most important example of a sheaf of $\cO_S$-Lie algebras on
$X$ is the relative tangent sheaf $T_{X/S}=\cDer _{\cO_S}(\cO_X,
\cO_X)$ with a natural bracket given by $[D_1,D_2]=D_1D_2-D_2D_1$
for local $\cO_S$-derivations $D_1$, $D_2$ of $\cO_X$.

\medskip

\begin{Definition}
An \emph{$\cO_S$-Lie algebroid on $X$} is a triple $(L, [\cdot,
\cdot ]_L, \alpha )$ consisting of a sheaf of $\cO_S$-Lie algebras
$(L, [\cdot, \cdot ]_L)$ on $X$  and a homomorphism $\alpha: L\to
T_{X/S}$, $x\to \alpha_x$, of sheaves of $\cO_S$-Lie algebras on
$X$, which satisfies the following Leibniz rule
$$[x,fy]_L=\alpha _x (f)y+f\,[x,y]_L$$
for all local sections $f\in \cO _X$ and $x,y\in L$ (in the
formula we treat $\alpha _x$ as an $\cO_S$-derivation of $\cO_X$).
We say that $L$ is \emph{smooth} if it is coherent and locally
free as an $\cO_X$-module. $L$ is \emph{quasi-smooth} if it is
coherent and torsion free as an $\cO_X$-module.
\end{Definition}

The map $\alpha$ in the above definition is usually called
\emph{the anchor}. A Lie algebroid is a sheaf of Lie-Rinehart
algebras (see \cite{Ri}). It is also a special case of the more
general notion of a Lie algebra in a topos defined by Illusie
(see \cite[Chapitre VIII, Definition 1.1.5]{Ill2}).

A \emph{homomorphism of $\cO_S$-Lie algebroids} $L$ and $L'$ on
$X$ is a homomorphism $L\to L'$ of sheaves of $\cO_S$-Lie algebras
on $X$ which commutes with the anchors.

Note that an $\cO_S$-Lie algebroid on $X$ with the zero anchor map
corresponds to a sheaf of $\cO_X$-Lie algebras.

\medskip

\begin{Definition}
A \emph{de Rham complex on $X$ over $S$} is a pair $(\bigwedge ^{\bullet}
M,d_M^{\bullet})$ consisting of the exterior algebra $\bigwedge ^{\bullet}
M:= \bigwedge ^{\bullet}_{\cO_X}M$
of an $\cO_X$-module $M$ and an $\cO_S$-anti-derivation
$d_M^{\bullet}: \bigwedge ^{\bullet} M \to \bigwedge ^{\bullet} M$
of degree $1$ (i.e., $d_M^{\bullet}(x\wedge y)=(d_M^{\bullet}x)\wedge y+(-1)^jx\wedge
d_M^{\bullet}y$ for all local sections $x\in \bigwedge^jM$ and $y\in
\bigwedge^{\bullet} M$) such that $(d_M^{\bullet})^2=0$.
We say that  $(\bigwedge ^{\bullet} M,d_M^{\bullet})$ is \emph{smooth} if $M$ is coherent
and locally free.
\end{Definition}

A de Rham complex is a special case of a sheaf of
graded-commutative differential graded algebras. A special case of
a de Rham complex is the de Rham complex $(\Omega_{X/S}^{\bullet},
d_{X/S}^{\bullet})$, which is the unique de Rham complex extending
the canonical $\cO _S$-derivation $d_{X/S}: \cO_X\to \Omega_{X/S}$
(uniqueness follows because $\Omega_{X/S}$ is generated by
$d_{X/S}\cO_X$ as a left $\cO_X$-module). By the universal
property of $d_{X/S}$ we have $\Der _{\cO_S}(\cO_X,M)\simeq \Hom
_{\cO_X}(\Omega_{X/S},M)$ and hence for every de Rham complex
$(\bigwedge ^{\bullet}M,d_M^{\bullet})$ we have a unique morphism of de Rham complexes
$(\Omega_{X/S}^{\bullet},d_{X/S}^{\bullet})\to (\bigwedge ^{\bullet}M,d_M^{\bullet})$.
This morphism induces a well defined map on the hypercohomology groups:
$$H^i_{DR}(X/S):=\HH ^{\bullet}(\Omega_{X/S}^{\bullet})\to \HH^{\bullet}({\bigwedge}^{\bullet}M).$$

\medskip

To every $\cO_S$-Lie algebroid  $(L, [\cdot, \cdot ]_L, \alpha )$
on $X$  we can associate a de Rham complex $(\bigwedge
^{\bullet}M,d_M^{\bullet})$ on $X$ over $S$ for $M=L^*$. This is
done by the following well known formula generalizing the usual
exterior differential:
$$\begin{array}{rcl}
(d_Mm) (l_1,...,l_{k+1})&=&\sum _{i=1}^{k+1}(-1)^{i+1}\alpha
_{l_i}(m(l_1,...,\hat{l_i}, ..., l_{k+1} ))\\
&+&\sum _{1\le i<j\le k+1}(-1)^{i+j} m([l_i,l_j]_L,
l_1,...,\hat{l_i}, ...,\hat{l_j}, ..., l_{k+1})
\end{array}
$$ for $m\in \bigwedge^k M$ and $l_1,..., l_{k+1}\in L$.
This gives a functor from the category of Lie algebroids to the
category of de Rham complexes.

On the other hand, to every de Rham complex $(\bigwedge^{\bullet}
M,d_M^{\bullet})$ on $X$ over $S$ we can associate a Lie algebroid
structure on $L=M^*$.  The anchor $L\to T_{X/S}=(\Omega_{X/S})^*$
is obtained as the transpose  of the $\cO_X$-homomorphism
$\Omega_{X/S}\to M$ corresponding to the $\cO_S$-derivation
$d_M:\cO_X\to M$. The bracket on $L$ can be read off the above
formula defining $d_M: M\to \bigwedge^2M$. This provides a functor
in the opposite direction: from  the category of de Rham complexes
to the category of Lie algebroids. These functors are
quasi-inverse on subcategories of smooth objects.

\medskip

If $L$ is a smooth $\cO_S$-Lie algebroid on $X$ then the
corresponding de Rham complex is denoted by
$(\Omega_{L}^{\bullet},d_{L}^{\bullet})$. In this case we set
$$H^i_{DR}(L):=\HH ^i(\Omega_{L}^{\bullet},d_{L}^{\bullet}).$$
We have the following standard spectral sequence associated to
the de Rham complex of $L$:
$$E^{ij}_1=H^j(X/S, \Omega_{L}^{i})\Rightarrow H^{i+j}_{DR}(L).$$

\subsection{Universal enveloping algebra of differential
operators}\label{univ-env-Lie-algebroid-construction}

\begin{Definition}
A \emph{sheaf of $\cO _S$-Poisson algebras on} $X$ is a pair
$(\cA, \{ \cdot, \cdot \})$ consisting of a sheaf $\cA$ of
commutative, associative and unital $\cO_X$-algebras with a
Poisson bracket $\{ \cdot, \cdot \}$ such that $(\cA, \{ \cdot,
\cdot \})$ is a sheaf of $\cO_S$-Lie algebras on $X$ satisfying
the Leibniz rule
$$\{x, y\cdot z\} =\{x,y\} \cdot z +y\cdot \{x,z\}$$
for all $x,y,z\in \cA$.
\end{Definition}

Let $\Lambda$ be a sheaf of rings of differential operators on $X$
over $S$ such that $\Lambda _0=\cO_X$. Let us assume that
$\Lambda$ is \emph{almost commutative}, i.e., the associated
graded $\Gr (\Lambda) $ is a sheaf of commutative
$\cO_X$-algebras. Then $\Gr (\Lambda) $ has a natural structure of
a sheaf of $\cO_S$-Poisson algebras on $X$ with the Poisson
bracket given by
$$\{[x], [y]\} :=\left(xy-yx\quad \mod \, {\Lambda} _{i+j-2} \right)\in  {\Gr} _{i+j-1} (\Lambda),$$
where $[x]\in  \Gr _i (\Lambda)$ is the class of $x\in \Lambda _i$
and  $[y]\in \Gr _j (\Lambda)$ is the class of $y\in \Lambda_j$.
The Poisson bracket induces an $\cO_S$-Lie algebroid structure on
$\Gr_1(\Lambda)$. The Lie bracket on $\Gr_1(\Lambda)$ is equal to
the Poisson bracket and the anchor map $\alpha: \Gr_1(\Lambda) \to
T_{X/S}$ is given by sending $[x]$ to the $\cO_S$-derivation $y\to
\{[x], y\}$, $y\in \cO_X=\Gr_{0}(\Lambda)$.

On the other hand, if $L$ is an $\cO_S$-Lie algebroid on $X$ then
we can associate to $L$ a sheaf of rings of differential operators
on $X$ over $S$ in the following way.
 We define an $\cO_S$-Lie algebra structure on $\ti L=\cO_X\oplus L$ by setting
$$[f+x, g+y]_{\ti L}=\alpha _x(g)-\alpha_y(f)+[x,y]_L$$
for all local sections $f,g\in \cO_X$ and $x,y\in L$. Let
$\cU_{\cO_S} (\ti L)$ be the universal enveloping algebra of $\ti
L$ and let $\ti \cU _{\cO_S}(\ti L)$ be the sheaf of subalgebras
(without unit!) generated by the image of the canonical map
$i_{\ti L}: \ti L\to \cU_{\cO_S} (\ti L)$ (note that in general
this map need not be injective). We define $\Lambda_L$ as the
quotient of $\ti \cU_{\cO_S}(\ti L)$ by the two-sided ideal
generated by all elements of the form $i_{\ti L}(f) i_{\ti
L}(x)-i_{\ti L}(fx)$ for all $f\in \cO_X$ and $x\in \ti L$. Let
$\Lambda_{L,j}$ be the left $\cO_X$-submodule of $\Lambda_L$
generated by products of at most $j$ elements of the image of $L$
in $\Lambda_L$. This defines a filtration of $\Lambda_L$ equipping
it with structure of sheaf of rings of differential operators
(since the canonical graded morphism $S^{\bullet} \Gr_1
(\Lambda_L)\to \Gr (\Lambda_L)$ is surjective, the constructed
$\Lambda_L$ is almost commutative). We call $\Lambda _L$ \emph{the
universal enveloping algebra of differential operators associated
to $L$}.

By the Poincare-Birkhoff-Witt theorem, if the Lie algebroid $L$ is
smooth then $L\to \Gr_1(\Lambda _L)$ is an isomorphism and the
canonical epimorphism $S^{\bullet} L \to \Gr (\Lambda_L)$ is an
isomorphism of sheaves of graded $\cO_X$-algebras (see
\cite[Theorem 3.1]{Ri}). This implies that if $L$ is quasi-smooth
then the canonical map $L\to \Lambda _L$ is injective.

If $L=T_{X/S}$ and the anchor map is identity, then $\Lambda_L$ is
denoted by $\cD_{X/S}$ and it is called \emph{the sheaf of
crystalline differential operators} (see \cite{BMR}). In \cite{BO}
the authors call it the sheaf of PD differential operators. In the
characteristic zero case the sheaf $\Lambda _L$ and the
correspondence between Lie algebroids and sheaves of rings of
differential operators was studied by Simpson in \cite[Theorem
2.11]{Si3} with subsequent corrections by Tortella in
\cite[Theorem 4.4]{To}.

\medskip

We can also consider twisted versions of sheaves of rings of
differential operators associated to a Lie algebroid (see
\cite{BB} and \cite{To}).

Let $\Lambda$ be an almost commutative  sheaf of rings of
differential operators on $X$ over $S$ such that $\Lambda
_0=\cO_X$. Then $\Lambda _1$ has an $\cO_S$-Lie algebra structure
on $X$ given by the usual Lie bracket $[\cdot, \cdot]$ coming from
$\Lambda$ and the anchor map given by sending $x\in \Lambda _1$ to
$f\to [x,f]$. Then $\Lambda _1\to \Gr _1(\Lambda)$ is a
homomorphism of $\cO_S$-Lie algebras with kernel being the sheaf
$\cO_X$ (with a trivial $\cO_S$-Lie algebroid structure).

The following definition is motivated by \cite[Definition
2.1.3]{BB}:

\begin{Definition}
A \emph{generalized $\cO_S$-Picard Lie algebroid} on $X$ is an
$\cO_S$-Lie algebroid $\ti L$ equipped with a section $1_{\ti L}$
of $\ti L$ inducing an exact sequence of $\cO_S$-Lie algebroids
$$0\to \cO_X \to \ti L\to
L\to 0,$$ where $\cO_X$ is taken with the trivial $\cO_S$-Lie
algebroid structure.
\end{Definition}

To any generalized $\cO_S$-Picard Lie algebroid $\ti L$ we can
associate an almost commutative sheaf of rings of differential
operators $\ti \Lambda _{\ti L}$ on $X$ over $S$ such that $\ti
\Lambda _{\ti L, 0}=\cO_X$ and $\ti \Lambda _{\ti L, 1}=\ti L$.
$\ti \Lambda _{\ti L}$ is constructed as a quotient of the
universal enveloping algebra of differential operators
$\Lambda_{\ti L}$ by the two-sided ideal generated by $1_{\ti
L}-1$. As in \cite[Lemma 2.1.4]{BB}, this defines a fully faithful
functor from the category of generalized Picard Lie algebroids to
the category of almost commutative sheaves of rings of
differential operators.

The analogous construction can be also found in \cite{To}, where
the author constructs $\ti \Lambda _{\ti L}$ by gluing local
pieces.

\section{Modules over Lie algebroids}

\subsection{Modules with generalized connections}

Let $X$ be an $S$-scheme. Let $M$ be a coherent $\cO_X$-module
with an $\cO_S$-derivation $d_M:\cO_X\to M$. A
\emph{$d_M$-connection} on a coherent $\cO_X$-module $E$ is an
$\cO_S$-linear morphism $\nabla: E\to E\otimes _{\cO_X} M $
satisfying the following Leibniz rule
$$\nabla (fe) =f\nabla(e)+e\otimes d_M(f) $$
for all local sections $f\in \cO_X$ and $e\in E$.

Note that notion of $d_M$-connection depends on the choice of
derivation $d_M$ and not only the sheaf $M$. For example if
$M=\Omega_{X/S}$ then  the standard derivation $d_{X/S}$ leads to
a sheaf with a usual connection whereas the zero derivation leads
to a Higgs sheaf (but without any integrability condition).

\subsection{Generalized Higgs sheaves}

Assume that $(\bigwedge ^{\bullet} M, d_{M}^{\bullet})$ is a de Rham complex and let
$E$ be a coherent $\cO_X$-module. Then a $d_M$-connection $\nabla:
E\to E\otimes M$ can be extended to a morphism $\nabla_i:
E\otimes_{\cO_X} {\bigwedge}^i M\to
E\otimes_{\cO_X}\bigwedge^{i+1}M$ by setting
$$\nabla_i (e\otimes \omega)=e\otimes d_M\omega +(-1)^i\nabla (e)\wedge \omega ,$$
where $e\in E$ and $\omega \in  {\bigwedge}^i M$ are local
sections. As usually one can check that the \emph{curvature}
$K=\nabla_1\circ \nabla $ is $\cO_X$-linear and $\nabla_{i+1}\circ
\nabla_i (e\otimes \omega )=K(e)\wedge \omega$. We say that $(E,
\nabla)$ is \emph{integrable} if the curvature $K=0$. If  $(E,
\nabla)$ is integrable then the sequence
$$0\to E\mathop{\to}^{\nabla}E\otimes M \mathop{\to}^{\nabla _1}E\otimes {\bigwedge}^2 M\to ...$$
becomes a complex. The hypercohomology groups of this complex are denoted
by $H^i_{DR}(X,E):=\HH ^i (E\otimes \bigwedge ^{\bullet}M, \nabla)$.

\medskip

Let $\bigwedge ^{\bullet}M$ be the de Rham complex corresponding
to the exterior $\cO_S$-algebra of $M$ with zero anti-derivation
$d_M$. Then a coherent $\cO_X$-module with an integrable
$d_M$-connection $\theta: E\to E\otimes _{\cO_X} M$ is called an
\emph{$M$-Higgs sheaf}. The corresponding homomorphism $\theta$ is
$\cO_X$-linear and it is called an \emph{$M$-Higgs field} (or just
a Higgs field). A \emph{system of $M$-Hodge sheaves} is an
$M$-Higgs sheaf $(E, \theta)$ with decomposition $E=\bigoplus E^j$
such that $\theta : E^j \to E^{j-1}\otimes M$. For
$M=\Omega_{X/S}$ we recover the usual notions of a Higgs sheaf and
a system of Hodge sheaves.
\medskip

To be consistent with notation in the characteristic zero case,
the hypercohomology groups $\HH ^i (E\otimes \bigwedge
^{\bullet}M, \theta)$ of the complex associated to an $M$-Higgs
sheaf are denoted by  $H^i_{Dol}(X,E)$. The following lemma can be
proven in the same way as \cite[Lemma 2.5]{Si2}:

\begin{Lemma}
Let $X$ be a smooth $d$-dimensional projective variety over an
algebraically closed field $k$ and let $(E, \theta)$ be an
$M$-Higgs sheaf. Then we have
$\chi _{Dol}(X, E)=\rk E \cdot \chi _{Dol}(X, \cO_X)$.
Moreover, if $E$ is locally free then
we have a perfect pairing
$$H^{i}_{Dol}(X, E)\otimes H^{2d-i}_{Dol}(X, E^*)\to k$$
induced by Serre's duality.
\end{Lemma}

\subsection{Modules over Lie algebroids and coHiggs sheaves}

Let $L$ be an $\cO_S$-Lie algebroid on $X$ and let $E$ be an
$\cO_X$-module. Let us recall that a (left) $\Lambda_L$-module
structure on $E$ is the same as an $L$-module structure, i.e., a
homomorphism  $\nabla: L\to \cEnd _{\cO_S}E$ of sheaves of
$\cO_S$-Lie algebras on $X$ (in particular,  $\nabla$ is
$\cO_X$-linear) satisfying Leibniz's rule
$$\nabla (x) (fe)=\alpha _x(f)e+\nabla(fx)(e)$$
for all local sections $f\in \cO_X$, $x\in L$ and $e\in E$. One
can also look at $L$-modules $E$ as modules $E$ over the sheaf of
$\cO_S$-Lie algebras $\ti L=\cO_X\oplus L$ on $X$ defined in
Subsection \ref{univ-env-Lie-algebroid-construction}, which
satisfy equality $(fy)e=f(ye)$ for all local sections $f\in
\cO_X$, $y\in L'$ and $e\in E$.

Proof of the following easy lemma is left to the reader:

\begin{Lemma}
Let $L$ be a smooth $\cO_S$-Lie algebroid $L$ and let $(\bigwedge
^{\bullet} \Omega_L, d_{\Omega_L}^{\bullet})$ be the associated de
Rham complex. Then we have an equivalence of categories between
the category of $L$-modules and coherent $\cO_X$-modules with
integrable $d_{\Omega_L}$-connection.
\end{Lemma}

Let $L$ be a coherent $\cO_X$-module. Let us provide it with the
trivial $\cO_S$-Lie algebroid structure, i.e.,  we take zero
bracket and zero anchor map. In this case we say that $L$ is
\emph{a trivial Lie algebroid}. For a trivial Lie algebroid the
corresponding sheaf of rings of differential operators
$\Lambda_{L}$ is equal to the (commutative) symmetric
$\cO_X$-algebra $S^{\bullet} (L)$. In this case \emph{an
$L$-coHiggs sheaf} is a (left) $\Lambda_L$-module, coherent as an
$\cO_X$-module. If $L$ is smooth then giving an $L$-coHiggs sheaf
is equivalent to giving an $\Omega_L$-Higgs sheaf.

If $L$ is smooth then $\VV (L)\to X$ is a vector bundle and we can
take its projective completion $\pi : Y=\PP (L\oplus \cO_X)\to X$.
The divisor at infinity $D=Y-\VV (L)$ is canonically isomorphic to
$\PP (L)$. On $Y$ we have the tautological relatively ample line
bundle $\cO _{\PP (L\oplus \cO_X) }(1)$. If $\cO_X (1)$ is an
$S$-ample polarization on $X$ then for sufficiently large $n$ the
line bundle $\cA=\cO _{\PP (L\oplus \cO_X) }(1)\otimes
\pi^*(\cO_X(n))$ is also $S$-ample.

By definition any $L$-coHiggs sheaf gives rise to a coherent
$\cO_{\VV (L)}$-module. The following lemma  describes image of
the corresponding functor (cf. \cite[Lemma 6.8 and Corollary
6.9]{Si4}):

\begin{Lemma}
We have an equivalence of categories between $L$-coHiggs sheaves
and coherent sheaves on $Y$, whose support does not intersect $D$.
Under this equivalence pure sheaves correspond to pure sheaves of
the same dimension and the notions of (semi)-stability are the
same when considered with respect to polarizations $\cO_X(1)$ on
$X$ and $\cA$ on $Y$.
\end{Lemma}

This lemma suggests another construction of the moduli space
$M^L_{\rm Dol}(X/S,P)=M^{\Lambda_L}(X/S, P)$ of Gieseker
semistable $L$-coHiggs sheaves (with fixed Hilbert polynomial $P$)
on $X/S$ using construction of the moduli space $M(Y/S, P)$ of
Gieseker semistable sheaves of pure dimension $n=\dim (X/S)$ on
$Y/S$ (with Hilbert polynomial $P$). Namely, $M(Y/S, P)$ is
constructed as a GIT quotient $R/\!\!/G$, where $R$ is some
parameter space and $G$ is a reductive group acting on $R$. Then
$M^L_{\rm Dol}(X/S,P)$ can be constructed as the quotient $R'/\!\!
/G$, where $R'$ is the $G$-invariant subscheme of $R$
corresponding to subsheaves whose support does not intersect $D$.

\subsection{Modules on varieties over fields}

In this subsection we take as $S$ the spectrum of an algebraically
closed field $k$. We also assume that $X$ is normal and projective
with fixed polarization $\cO_X(1)$.

We say that a sheaf with an $M$-connection $(E, \nabla)$ is
\emph{slope semistable} if $E$ is torsion free as an
$\cO_X$-module and if for any $\cO_X$-submodule $E'\subset E$ such
that $\nabla (E')\subset E'\otimes _{\cO_X} M$ we have
$$\mu (E')\le \mu (E).$$
We say that $(E, \nabla)$ is \emph{slope stable} if we have
stronger inequality $\mu (E')< \mu (E)$ for every proper
$\cO_X$-submodule $E'\subset E$ preserved by $\nabla$ and such
that $\rk E'<\rk E$. In much the same way we can introduce notions
of slope (semi)stability for $M$-Higgs sheaves and systems of
$M$-Hodge sheaves. In each case to define (semi)stability we use
only subobjects in the corresponding category.

\medskip

Let us fix a smooth $k$-Lie algebroid $L$ on $X$. We have a
natural action of $\Gm$ on $\Omega_L$-Higgs sheaves given by
sending $(E, \theta)$ to $(E, t\theta)$ for $t\in \Gm$. The
following lemma is a simple generalization of the well known fact
in case of usual Higgs bundles (see, e.g., \cite[Lemma 4.1]{Si2})
but we include proof for completeness.  The assertion in the
positive characteristic case is slightly different to that of
\cite[Lemma 4.1]{Si2}. The difference comes from the fact that for
$k=\bar{\FF} _p $ every $t\in k^*$ is a root of unity.

\begin{Lemma}
A rank $r$ torsion free $\Omega_L$-Higgs sheaf $(E, \theta)$ is a
fixed point of the $\Gm$-action if and only if it has a structure
of system of $\Omega_L$-Hodge sheaves.
\end{Lemma}

\begin{proof}
Taking reflexivization we can assume that $E$ is reflexive. By
assumption for every $t\in \Gm$ there exists an isomorphism of
$\cO _X$-modules $f:E\to E$ (depending on $t$) such that
$f\theta=t\theta f$. On the subset $U$ where $E$ is locally free,
the coefficients of the characteristic polynomial of $f$ define
sections of $\cO _X$. Since $X$ is normal and projective we have
$\cO _X(U)=\cO_X(X)=k$, so they are constant. Hence we can
decompose $E$ into eigensubsheaves $E=\bigoplus E_{\lambda}$,
where $E_{\lambda}=\ker (f-\lambda)^r$ for $\lambda \in k^*$
(eigenvalue $0$ does not occur as $f$ is an isomorphism). Since
$(f-t\lambda)^r\theta = t^r \theta (f-\lambda)^r$, the Higgs field
$\theta$ maps $E_{\lambda}$ to $E_{t\lambda}$. If we take $t$ such
that $t^j\ne 1$ for $j=0,..., r $ then for every eigenvalue
$\lambda$ the elements $\lambda, t\lambda,..., t^r\lambda$ are
pairwise distinct. So there exists $j_0$ such that
$t^{j_0}\lambda$ is an eigenvalue but $t^{j_0-1}\lambda$ is not an
eigenvalue. Then $E^i=\bigoplus_{j_0\le j\le i} E_{t^j\lambda}$
defines a system of $\Omega_L$-Hodge sheaves which is a direct
summand of $(E, \theta)$. So we can complete the proof by
induction on the rank $r$ of $E$.
\end{proof}

\begin{Corollary} \label{Higgs=Hodge}
A system of $\Omega_L$-Hodge sheaves $(E, \theta)$ is slope (or
Gieseker) semistable if and only if it is slope (respectively,
Gieseker) semistable as an $\Omega_L$-Higgs sheaf.
\end{Corollary}

\begin{proof}
It is sufficient to prove that the maximal destabilizing
$\Omega_L$-Higgs subsheaf of a system of $\Omega_L$-Hodge sheaves
$(E, \theta)$ is a system of $\Omega_L$-Hodge sheaves. This
follows from the above lemma and the fact that the maximal
destabilizing $\Omega_L$-Higgs subsheaf is unique so it is
preserved by the natural $\Gm$-action.
\end{proof}

\subsection{Hitchin's morphism for moduli spaces of $L$-coHiggs
sheaves}

Let $G$ be a quasi-coherent $\cO_S$-module. Consider the functor
which to an $S$-scheme $T$ associates $\Hom _{\cO_T}(G_T, \cO_T)$.
It is representable by the $S$-scheme $\VV (G)$. In particular,
for $\pi :T=\VV (G)\to S$ we get the tautological homomorphism
$$\lambda_G\in \Hom _{\cO_
{\VV (G)}}(\pi^*G, \cO_ {\VV (G)})=\Hom _{\cO_S}(G, \pi _*\cO_
{\VV (G)})=\Hom _{\cO_S-\rm{alg}}(S^{\bullet}G, S^{\bullet}G)$$
corresponding to the identity on $S^{\bullet}G$.

If $G$ is a locally free sheaf of finite rank then $\VV (G)\to S$
is a vector bundle with sheaf of sections isomorphic to $G^*$.

\medskip

The following lemma was explained to the author by C. Simpson:

\begin{Lemma}\label{Simpson}
Let $f:X \to S$ be a flat projective morphism of noetherian schemes and
let $G$ be a locally free sheaf on $X$. Then the functor
$\underline{H}^0(X/S, G)$ which to an $S$-scheme $h: T\to S$
associates $H^0(X_T/T, G_T)$ is representable by an $S$-scheme.
\end{Lemma}

\begin{proof}
Since certain twist of $G^*$ by a relatively very ample line bundle
is relatively globally generated, we can embed $G$ as a subbundle into a
direct sum $K_1$ of relatively very ample line bundles. Then we can again embed
the quotient $K_1/G$ into $K_2$ with $K_2$ a direct sum of relatively very ample
bundles. Then for any $S$-scheme $T$ we have an exact sequence
$$0 \to \underline{H}^0(X/S, G)(T) \to \underline{H}^0 (X/S, K_1 )(T) \to \underline{H}^0(X/S, K_2)(T).$$
But we can assume that all the higher direct images of $K_1$ vanish and then by the Grauert's
theorem $\underline{H}^0 (X/S, K_1 )$ is representable by the bundle $\VV (f_*K_1)\to S$.
Similarly, $\underline{H}^0 (X/S, K_1 )$ is representable by the bundle $\VV (f_*K_2)\to S$.
Therefore $\underline{H}^0(X/S, G)$  is represented by the kernel of the map between bundles.
This is a vector subscheme of $\VV ( f_*K_1 )\to S$.
\end{proof}

\medskip
We will also need the following well-known lemma:

\begin{Lemma}\label{extension}
Let $f: X\to S$ be a flat family of irreducible $d$-dimensional
schemes satisfying Serre's condition $(S_2)$. Let $E$ be an
$S$-flat coherent $\cO_X$-module such that $E\otimes k(s)$ is
pure of dimension $d$ for every point $s\in S$.
Then there exists a relatively big open subset
$j: U\subset X$ such that $E^{**}\to j_*(E|_U)$ is an isomorphism.
\end{Lemma}

\medskip

Consider a flat projective morphism $f: X\to S$ of noetherian schemes.
Let $L$ be a smooth $\cO_S$-Lie algebroid on $X$ and let us recall that
$\Omega_L=L^*$. Consider the functor which to an $S$-scheme $h:
T\to S$ associates
$$\bigoplus _{i=1}^rH^0(X_T/T, S^i\Omega_{L,T}).$$
By Lemma \ref{Simpson} this functor is representable by an
$S$-scheme $\VV ^L  (X/S, r)\to S$.

Let us also assume that $X/S$ is a family of $d$-dimensional
varieties satisfying Serre's condition $(S_2)$. If $T$ is an
$S$-scheme then $X_T/T$ is also a flat family of $d$-dimensional
varieties satisfying Serre's condition $(S_2)$.

Assume that $L$ is a trivial $\cO_S$-Lie algebroid and consider a
family $(E, \theta : E\to E\otimes \Omega_{L,T})$ of $L$-coHiggs
sheaves of pure dimension $d=\dim (X/S)$ on the fibres of $X_T\to
T$. Then there exists an open subset $U\subset X_T$ such that $E$
is locally free on $U$ and the intersection of $U$ with any fiber
of $X_T\to T$ has a complement of codimension at least $2$. Let us
consider ${\bigwedge}^i\theta|_U: {\bigwedge}^i (E|_U)\to
{\bigwedge}^i (E|_U\otimes _{\cO _U} \Omega_{L,T}|_U)$. We have a
well defined surjection ${\bigwedge}^i ( E|_U\otimes _{\cO _U}
\Omega_{L,T}|_U) \to {\bigwedge}^i E|_U\otimes _{\cO_U}
S^i\Omega_{L,T}|_U$, given by
$$(e_1\otimes \lambda_1 )\wedge ...\wedge (e_i\otimes \lambda_i) \to
(e_1\wedge...\wedge e_i)\otimes (\lambda_1...\lambda_i),$$
where $e_1,..., e_i\in E$ and $\lambda_1,..., \lambda_i\in
\Omega_{L,T}$. So we get a morphism of sheaves
$$\cO_U\to \cEnd _{\cO_X}({\bigwedge}^iE)|_U\otimes_{\cO_U}S^i\Omega_{L,T}|_U
\mathop{\longrightarrow}^{(-1)^i\Tr \otimes \id}
S^i\Omega_{L,T}|_U$$ The corresponding section $\sigma_i
(\theta|_U)\in H^0(U, S^i\Omega_{L,T}|_U)$ is just an evaluation
of the $i$-th elementary symmetric polynomial on $\theta|_U$. By
Lemma \ref{extension} this section extends uniquely to section
$\sigma_i (\theta )\in H^0(X_T/T, S^i \Omega_{L,T})$. In this way
we can define a $T$-point $\sigma(E,\theta)=(\sigma_1
(\theta),...,\sigma_r (\theta ))$ of $\VV ^L(X/S, r)$.

Let $P$ be a polynomial of degree $d=\dim (X/S)$ corresponding to
(some) rank $r$ torsion free sheaves on the fibres of $X\to S$.
Consider the moduli space $M _{\rm Dol}^L(X/S, P)$ of Gieseker
semistable $L$-coHiggs sheaves with Hilbert polynomial $P$. Then
the above construction defines a morphism of functors inducing the
corresponding morphism of coarse moduli spaces $H_L:M_{\rm
Dol}^L(X/S, P)\to \VV ^L(X/S, r)$. This morphism is called
\emph{Hitchin's morphism}.

There is also a stack theoretic version of Hitchin's morphism. The
moduli stack of $L$-coHiggs sheaves is defined as a lax functor
between $2$-categories by
$$\begin{array}{cccc}
\cM^L_{\rm Dol}(X/S, P): &(\Sch /S)&\to& (\hbox{groupoids})\\
&T&\to&\cM(T),
\end{array}$$
where $\cM(T)$ is the category whose objects are $T$-flat families
of pure $d$-dimensional $L$-coHiggs sheaves with Hilbert
polynomial $P$ on the fibres of $X_T\to T$, and whose morphisms
are isomorphisms of coherent sheaves. Then $\cM^L_{\rm Dol}(X/S,
P)$ is an algebraic stack for the fppf topology on $(\Sch /S)$. As
above we can construct Hitchin's morphism $\cM^L_{\rm Dol}(X/S, P)
\to \VV ^L(X/S, r)$. By abuse of notation, we also denote this
morphism by $H_L$.

\medskip

As in the usual Higgs bundle and characteristic zero case, one can
construct the total spectral scheme $\WW ^L(X/S,r)\subset
\VV(L)\times _S \VV ^L(X/S, r)$, which is finite and flat over
$X\times _S \VV ^L (X/S, r)$. This subscheme has the property that
for any family $(E, \theta : E\to E\otimes \Omega_{L,T})$ of
$L$-coHiggs sheaves of pure dimension $d$ on the fibres of $X_T\to
T$, the corresponding coherent sheaf on $\VV (L_T)$ is
set-theoretically supported on $\WW ^L(X/S,r)\times _{\VV
^L(X/S,r)}T$. This can be seen as follows. Let $x$ be a geometric
point of $X$ at which $E$ is locally free. Then
$S^{\bullet}L\otimes k(x)$ acts on $V=E\otimes k(x)$ via $\theta
(x)$. Let us recall that over an algebraically closed field any
finitely dimensional vector space which is irreducible with
respect to a set of commuting linear maps has dimension $1$.
Therefore $V$ has a filtration $0=V_0\subset V_1\subset ...\subset
V_r=V$ with quotients $V^i=V_i/V_{i-1}$ of dimension $1$ over
$k(x)$ and such that $\theta(x)$ acts on $V^i$ as multiplication
by $\lambda_i\in (L\otimes k(x))^*$. It is clear from our
definition that $\tau \in L\otimes k(x)$ acts on $V$ via
$\theta_\tau :=\theta (x) ^T(\tau)$ in such a way that in the
characteristic polynomial
$$\det (t\cdot I-\theta_{\tau})=t^r+\sigma_1(\theta_\tau)t^{r-1}
+...+\sigma_r(\theta_\tau)$$ we have $\sigma_i(\theta_\tau)=(-1)^i
\sum_{1\le j_1<...<j_i\le r}\lambda_{j_1}...\lambda_{j_i}$. This
and the Cayley--Hamilton theorem show that the coherent sheaf on
$\VV (L_T)$ corresponding to $(E,\theta)$ has a scheme-theoretic
support contained in $\WW ^L(X/S,r)\times _{\VV^L(X/S,r)}T$ and it
coincides with it set-theoretically.

\medskip

Note that in the curve case there exists a different
interpretation of Higgs bundles using cameral covers. Such an
approach allows to deal with general reductive groups (see
\cite{DG} for the characteristic zero case). In positive
characteristic the analogous construction requires some
restrictions on the characteristic of the base field.

\medskip

The following theorem can be proven in a similar way as the usual
characteristic zero version \cite[Theorem 6.11]{Si4}. It also
follows from Langton's type Theorem \ref{Langton}.

\begin{Theorem}
Hitchin's morphism $H_L: M_{\rm Dol}^L(X/S, P)\to \VV ^L(X/S, r)$
is proper.
\end{Theorem}

\subsection{Deformation of a Lie algebroid over an affine line.\label{deformation-over-line}}

Let $R$ be a commutative ring with unity. Let $f: X\to S$ be a
morphism of $R$-schemes. Let $\AA ^1_R=\Spec R[t]$ and let
$p_1:X\times _R\AA ^1_R\to X$ be the projection onto the first
factor.

Let us consider an $\cO_S$-Lie algebroid $L$ on $X$ and the
morphism $f\times \id :X\times _R \AA ^1 _R\to S\times _R \AA
^1_R$ of $R$-schemes. We can define an $\cO_{S\times _R\AA
^1_R}$-Lie algebroid $L^R$ on $X\times _R\AA ^1 _R$  by taking $
L^R:=p_1^*L$ with Lie bracket given by $[\cdot ,
\cdot]_{L^R}:=p_1^*[\cdot , \cdot]_L\otimes t $ and the anchor map
given by $\alpha ^R:=p_1^*\alpha \otimes t$.

The universal enveloping algebra of differential operators
$\Lambda^R_L:=\Lambda _{L^R}$ associated to $L^R$ can be
constructed as a subsheaf of $p_1^*\Lambda_L$ generated by
sections of the form $\sum t^i\lambda_i$, where $\lambda_i$ are
local sections of $\Lambda_{L,i}$.

If $R=k$ is a field and $s\in \AA^1(k)-\{0\}$ then the restricted
sheaf $\Lambda _L^R|_{X\times \{s\}}$ is naturally isomorphic to
$\Lambda_L$. The sheaf $\Lambda _L^R|_{X\times \{0\}}$ is
naturally isomorphic to the associated graded sheaf of algebras
$\Gr \Lambda_L$. This gives a deformation of $\Lambda_L$ to its
associated graded sheaf of algebras (or a quantization of the
commutative algebra $\Gr \Lambda_L$).

\medskip

Let $T$ be an $S$-scheme and let us fix $\lambda \in
H^0(T/R,\cO_{T})$. Let $E$ be a coherent $\cO _{X_T}$-module and
let $p_X$ and $p_T$ be the projections of $X\times _ST$ onto $X$
and $T$, respectively. Let $(M, d_M)$ be a coherent $\cO_X$-module
with an $\cO_S$-derivation.

Then we set $\ti M=p_X^*M$ and $d_{\ti M}= p_X^*d_M\cdot
p_T^*\lambda$. A $d_{\ti M}$-connection on $E$ is called a
\emph{$\lambda $-$d_M$-connection}. This generalizes the usual
notion of $\lambda$-connection.

For the constant section $\lambda=0\in H^0(T/R,\cO_{T})$ an
integrable $\lambda$-$d_M$-connection is just an $M$-Higgs field.
Similarly, for $\lambda =1\in H^0(T/R,\cO_{T})$ we recover the
notion of a $d_M$-connection.

\medskip

Assume that $L$ is a smooth $\cO_S$-Lie algebroid on $X$. Let us
fix  a morphism of $R$-schemes $T\to S\times _R \AA ^1_R$ and let
$\lambda \in H^0(T/R,\cO_{T})$ be the section corresponding to the
composition of $T\to S\times _R \AA ^1_R$ with the canonical
projection $S\times _R \AA ^1_R\to \AA^1 _R$. Since $T\times_{
S\times _R \AA ^1_R}X\times _R \AA ^1_R=X_T$,  an $L^R$-module
structure on a coherent $\cO_{X_T}$-module $E$ is equivalent to
giving an integrable $\lambda$-$d_{\Omega_L}$-connection.

\section{Lie algebroids in positive characteristic}

\subsection{Sheaves of restricted Lie algebras}

Let $R$ be a commutative ring (with unity) of characteristic $p$
and let $L$ be a Lie $R$-algebra. We define the universal Lie
polynomials $s_j$ by the formula
$$s_j(x_1,x_2)=-{\frac{1}{j}} \sum _{\sigma} \ad
x_{\sigma(1)}...\ad x_{\sigma(p-1)}(x_2)$$ in which we sum over
all $\sigma: \{1,...,p-1\}\to \{1,2\}$ taking $j$ times value $1$.

Let $A$ be an associative $R$-algebra. For $x \in A$ we define
$\ad (x):A\to A $ by the formula $(\ad (x))(y)=xy-yx$ for $y\in
A$. Then we have the following well known Jacobson's formulas:
$$\ad (x^p)=\ad (x)^p$$
$$(x+y)^p=x^p+y^p+\sum _{0<j<p}s_j(x,y).$$

\medskip

Let $X$ be a scheme over a scheme $S$ of characteristic $p>0$. A
\emph{sheaf of restricted $\cO_S$-Lie algebras on $X$} is a sheaf
of $\cO_S$-Lie algebras $(L, [\cdot, \cdot ])$ on $X$ equipped
with a $p$-th power operation $L\to L$, $x\to x^{[p]}$, which
satisfies the following conditions:
\begin{enumerate}
\item $(fx)^{[p]}=f^px^{[p]}$ for all local sections $f\in \cO_S$
and $x\in L$,
\item $\ad (x^{[p]})=(\ad(x))^{p}$ for $x\in L$,
\item $(x+y)^{[p]}=x^{[p]}+y^{[p]}+\sum _{0<j<p}s_j(x,y)$ for all
$x,y\in L$.
\end{enumerate}

A homomorphism of sheaves of restricted $\cO_S$-Lie algebras
$\varphi: L\to L'$ on $X$ is such a homomorphism of sheaves of
$\cO_S$-Lie algebras on $X$ that $\varphi(x^{[p]})=\varphi(x)^{[p]}$ for
all $x\in L$.

\medskip

Let $\cA$ be a sheaf of associative $\cO_S$-algebras on $X$. It
has a natural structure of a sheaf of restricted $\cO_S$-Lie
algebras on $X$ with bracket $[x,y]=xy-yx$ and $p$-th power
operation $x^{[p]}=x^p$ for local sections $x, y\in \cA$.

Now let $L$ be a sheaf of restricted $\cO_S$-Lie algebras on $X$.
For any homomorphism $\varphi : L\to \cA$ of sheaves of
$\cO_S$-Lie algebras on $X$ we can define $\psi : L\to \cA$ by
$\psi (x)=(\varphi (x))^{p}-\varphi (x^{[p]})$ for $x\in L$. The map $\psi$
measures how far is $\varphi$ from being a homomorphism of sheaves of restricted
$\cO_S$-Lie algebras on $X$.

\begin{Lemma} \label{p-curvature-for-Lie-algebras}
The map $\psi : L\to \cA$ is additive and its image commutes with
the image of $\varphi$. In particular, $[\psi (L), \psi (L)]=0$.
\end{Lemma}

\begin{proof}
Let us take sections $x, y\in L(U)$ for some open subset $U\subset X$.
From Jacobson's formula in  $\cA$ we have
$$(\varphi(x+y))^p=\varphi (x)^p+\varphi (y)^p+\sum _{0<j<p}s_j(\varphi (x),\varphi (y))$$
On the other hand, from definition of a sheaf of restricted Lie algebras we have
$$\varphi((x+y)^{[p]})=\varphi (x^{[p]})+\varphi (y^{[p]})+\sum _{0<j<p}s_j(\varphi (x),\varphi (y)),$$
so subtracting these equalities we get  additivity of $\psi$.

Now we need to prove that
$[\psi (x), \varphi (y)]=0$.
But we have
$$[\varphi (x)^p, \varphi (y)]=\ad (\varphi (x)^p) (\varphi (y))=(\ad \varphi (x))^p (\varphi (y))$$
and
$$[\varphi (x^{[p]}), \varphi (y)]=\varphi([x^{[p]}, y])=\varphi (\ad (x^{[p]})(y))=
\varphi (\ad (x)^p (y))=(\ad \varphi (x))^p (\varphi (y)),$$
so subtracting yields the required equality.
\end{proof}

\medskip

The \emph{restricted universal enveloping algebra} $\cU ^{[p]}
_{\cO_S}(L)$ of a sheaf of restricted $\cO_S$-Lie algebras $L$ on
$X$ is the quotient of the universal enveloping algebra $\cU
_{\cO_S}(L)$ by the two-sided ideal generated by all elements of
the form $x^p-x^{[p]}$ for local sections $x\in L$.

If $S=X$ and $L$ is locally free as an $\cO_X$-module then $L$ is
contained in $\cU ^{[p]} _{\cO_X}(L)$. Moreover, if $x_1,...,x_r$
are local generators of $L$ as an $\cO_X$-module then
$x_1^{i_1}...x_r^{i_r}$ with $0\le i_j<p$ for all $j$, form a
local basis of $\cU ^{[p]} _{\cO_X}(L)$ as an $\cO_X$-module. In
particular, $\cU ^{[p]} _{\cO_X}(L)$ is locally free of rank
$p^{\rk L}$. In this case for  any sheaf $\cA$ of associative
algebras on $X$ and any homomorphism $\varphi : L\to \cA$ of
sheaves of Lie algebras on $X$, the map $\psi : L\to \cA $ is $F_X^*$-linear, i.e.,
$\psi (fx)=f^p \psi (x)$ for all $f\in \cO_X$ and $x\in L$
(this follows from the first condition in the definition of a sheaf of 
restricted Lie algebras). So  by adjunction $\psi$ induces an $\cO_X$-linear
map $F_X^*L\to \cA$ that by abuse of notation is also denoted by $\psi$.
Then the restricted universal enveloping algebra $\cU ^{[p]}_{\cO_X}(L)$
has the following universal property. For any sheaf $\cA$ of
associative $\cO_X$-algebras and any homomorphism $\varphi : L\to
\cA$ of sheaves of $\cO_X$-Lie algebras with $\psi : L\to \cA$
equal to zero, there exists a unique homomorphism $\ti \varphi :
\cU ^{[p]}_{\cO_X}(L)\to \cA$ of sheaves of associative
$\cO_X$-algebras such that $\varphi : L\to \cA$ is the composition
of the natural map  $L \to \cU ^{[p]}_{\cO_X}(L)$ with $\ti
\varphi$.

\subsection{Restricted Lie algebroids}

Note that the relative tangent sheaf $T_{X/S}$ has a natural
structure of a sheaf of restricted $\cO_S$-Lie algebras on $X$ in
which the $p$-th power operation on $\cO_S$-derivation $D:\cO_X\to
\cO_X$ is defined as the derivation acting on functions as the
$p$-th power differential operator $D^p$. In fact, $T_{X/S}$ with
the usual Lie bracket and this $p$-th power operation is a sheaf
of restricted $\cO_S$-Lie subalgebras of the associative algebra
$\cEnd _{\cO_S}\cO_X$ taken with the natural structure of a sheaf
of restricted $\cO_S$-Lie algebras on $X$. This motivates the
following definition:

\begin{Definition}
A \emph{restricted $\cO_S$-Lie algebroid} on $X$ is a quadruple
$(L, [\cdot, \cdot ], {\cdot}^{[p]}, \alpha )$ consisting of a
sheaf of restricted $\cO_S$-Lie algebras $(L, [\cdot, \cdot ],
{\cdot}^{[p]})$ on $X$ and  a homomorphism of sheaves of
restricted $\cO_S$-Lie algebras $\alpha: L\to T_{X/S}$ on $X$
satisfying the Leibniz rule and the following formula:
$$(fx)^{[p]}=f^px^{[p]}+\alpha ^{p-1}_{fx}(f) x$$
for all $f\in \cO_X$ and $x\in L$.
\end{Definition}

As in the non-restricted case we can define a \emph{trivial
restricted Lie algebroid} as a trivial Lie algebroid with the zero
$p$-th power operation. $T_{X/S}$ with the usual Lie bracket and
$p$-th power operation will be called the \emph{standard
restricted $\cO_S$-Lie algebroid on $X$}.

The last condition in the definition requires certain
compatibility of the $p$-th power operation on $L$ with the anchor
map and $\cO_X$-module structure of $L$. It can be explained by
the fact that, as expected, a restricted $\cO_S$-Lie algebroid on
$X$ with the zero anchor map is a sheaf of restricted $\cO_X$-Lie
algebras. In fact, the formula in the definition comes from the
following Hochschild's identity:

\begin{Lemma} \label{Hochschild}
\emph{(see \cite[Lemma 1]{Ho})}
Let $A$ be an associative $\FF_p$-algebra and $R\subset A$ a
commutative subalgebra. If for an element $x\in A$ we have $(\ad
x)(R)\subset R$ then for any element $r\in R$ we have
$$(rx)^p=r^px^p+(\ad (rx))^{p-1}(r) x.$$
\end{Lemma}

A similar formula can be found as \cite[Proposition 5.3]{Ka1}
(although with a sign error as pointed out by A. Ogus in
\cite{Og}).

\medskip

The following criterion allows us to check when a submodule of a
restricted Lie algebroid is a restricted Lie subalgebroid.
It generalizes well known Ekedahl's criterion allowing to check
when a submodule of the tangent bundle defines a $1$-foliation
(see \cite[Lemma 4.2]{Ek}).

\begin{Lemma}
\begin{enumerate}
\item
Let $L'$ be an $\cO_X$-submodule of an $\cO_S$-Lie algebroid $L$
on $X$. Then the Lie bracket on $L$ induces an $\cO_X$-linear
map
$${\bigwedge}^2L'\to L/L'$$
sending $x\wedge y$ to the class of $[x,y]$.
If this map is the zero map then $L'$ is an $\cO_S$-Lie subalgebroid of $L$.
\item
If $L'$ is an $\cO_S$-Lie subalgebroid of a restricted $\cO_S$-Lie algebroid $L$
then the $p$-th power map induces an $\cO_X$-linear morphism $F^*_{X}L'\to L/L'$.
If this map is the zero map then $L'$ is a restricted $\cO_S$-Lie subalgebroid
of $L$.
\end{enumerate}
\end{Lemma}

\begin{proof}
Let us take $f\in \cO_X$ and $x, y \in L'$. The first part follows
from the equality
$$[x, fy]=f[x,y]+\alpha_x(f)y\equiv f[x, y] \quad \mod \, L'.$$
To prove the second part note that
$$(x+y)^{[p]}=x^{[p]}+y^{[p]}+\sum _{0<j<p}s_j(x,y)\equiv x^{[p]}+y^{[p]} \quad \mod \, L' ,$$
since $s_j(x,y)\in L'$, as the $s_j$ are Lie polynomials.
Therefore $F^*_{X}L'\to L/L'$ is additive. Hence to prove that it
is $\cO_X$-linear it is sufficient to note that
$$(fx)^{[p]}=f^px^{[p]}+\alpha ^{p-1}_{fx}(f) x  \equiv f^px^{[p]} \quad \mod \, L'  .$$
\end{proof}

\medskip

Let us consider the following commutative diagram
$$ \xymatrix{ &{\VV (L)}^{(1/X)}\ar[rd]\ar[r]^{=}&
{\VV (F_{X/S}^*L')}\ar[d]^{\pi}\ar[r]^{{\ti F}_{X/S}} & \VV (L')\ar[r]\ar[d]^{\pi '}&\VV (L)\ar[d]\\
&&X\ar@(dr,dl)[rr]^{\mbox{\scriptsize $F_{X}$}}\ar[r]^{F_{X/S}}&X' \ar[r]&X\\
}$$ in which $L'$ is the pull back of  $L$ via $X'\to X$.

The following lemma is an analogue of \cite[Lemma 1.3.2]{BMR}:

\begin{Lemma} \label{BMR-lemma}
Let $L$ be a restricted $\cO_S$-Lie algebroid on $X$. Then the map
$\imath : L\to \Lambda_L$ sending $x\in L$
to $\imath(x):=x^p-x^{[p]}\in \Lambda_L $ is $F_X^*$-linear and its
image is contained in the center $Z(\Lambda_L)$ of $\Lambda_L$. In particular, if
$L$ smooth then $\imath$ extends to an $\cO_{X'}$-linear inclusion 
$S^{\bullet}L'\hookrightarrow F_{X/S,*}Z(\Lambda_L)$.
\end{Lemma}

\begin{proof}
Lemma \ref{Hochschild} proves that the $p$-th power operation satisfies 
$\alpha _{x^{[p]}}=(\alpha _x)^p$ and $(fx)^{[p]}-f^px^{[p]}=(fx)^p-f^px^p$ 
in $\Lambda _L$ for all $f\in \cO_X$ and $x\in L$. Hence $\imath$ is $F_X^*$-linear.
Lemma \ref{p-curvature-for-Lie-algebras} implies that its image is contained in 
$Z(\Lambda_L)$.

For any $f\in \cO_X$ and $x\in L$ we have $xf^p-f^px=\alpha_x(f^p)=0$ in $\Lambda_L$, as 
$\alpha_x$ is an $\cO_S$-derivation. Therefore 
$\cO_{X'}\subset F_{X/S,*}Z(\Lambda_L)$ which together 
with the first part proves the required assertion. 
\end{proof}

\medskip

Note that the above lemma shows that $\Lambda _L$ contains a commutative subalgebra
$ S^{\bullet}(F_{X/S}^*L')$, so $\Lambda_L$
defines a quasi-coherent sheaf ${\ti\Lambda}_{L}$ on $\VV(F_X^*L)$.

Let $\Lambda _L^{[p]}$ be the quotient of $\Lambda_L$ by the
two-sided ideal generated by $\imath (x)$ for $x\in L$.
We call it the \emph{restricted universal enveloping algebra of
differential operators of $L$}.

\medskip

\begin{Lemma}\label{rank-lemma}
Let $L$ be smooth of rank $m$. Then ${\ti\Lambda}_{L}$ is a
locally free $\cO_{\VV (L)}$-module of rank $p^{m}$.
\end{Lemma}

\begin{proof}
The canonical embedding $j: L\to \Lambda_L$ induces an embedding
$\ti j: L\to \Lambda_L^{[p]}$. Let us take an open subset
$U\subset X$ such that $L(U)$ is a free $\cO_X(U)$-module with
generators $x_1,..., x_m$. The kernel of $\Lambda _L(U)\to \Lambda
_L^{[p]}(U)$ is generated by elements $\imath (x_1),...,\imath
(x_m)$ which are in the center of $\Lambda _L (U)$. But $\imath
(x_i)\equiv x_i^p \, \mod \Lambda _{L, p-1}$, so by the
Poincare-Birkhoff-Witt theorem $\Lambda_L^{[p]}$ has local generators ${\ti
j}(x_1)^{i_1}...{\ti j}(x_m)^{i_m}$ for $0\le i_l<p$.  Hence
${j}(x_1)^{i_1}...{ j}(x_m)^{i_m}$ for $0\le i_l<p$ locally
generate $\Lambda_L$ as  an $S^{\bullet}(F_X^*L)$-module and
${\ti\Lambda}_{L}$ is locally free of rank $p^{m}$.
\end{proof}

\medskip

Lemma \ref{BMR-lemma} shows that if $L$ is smooth then $\imath$ induces 
an $\cO_{X'}$-linear map
$L'\to F_{X/S, *}\Lambda_L$ and a homomorphism of sheaves of
$\cO_{X'}$-algebras
$$S^{\bullet}(L')\to F_{X/S, *}(Z(\Lambda_L))\subset \Lambda_L':=F_{X/S, *}\Lambda_L .$$
In particular, it makes $\Lambda'_L$ into a quasi-coherent sheaf
of $S^{\bullet}(L')$-modules. This sheaf defines on $\VV(L')$ a
quasi-coherent sheaf of $\cO_{\VV (L')}$-algebras
${\ti\Lambda}_{L}'$. Note that by construction
$$\pi'_*{\ti\Lambda}_{L}'=F_{X/S, *}\Lambda_L =F_{X/S, *}\pi_*{\ti\Lambda}_{L}=\pi'_*
{\ti F}_{X/S, *}{\ti\Lambda}_L,$$ so we have
$${\ti\Lambda}_{L}'={\ti F}_{X/S, *}{\ti\Lambda}_L .$$
By an explicit computation as in Lemma \ref{rank-lemma} one can
prove the following theorem:

\begin{Theorem}
Assume that $X/S$ is smooth of relative dimension $d$ and $L$ is
smooth of rank $m$. Then ${\ti \Lambda}_L'$ is a locally free
$\cO_{\VV (L')}$-module of rank $p^{m+d}$.
\end{Theorem}

\medskip

By \cite{BMR} in the special case when $L=T_{X/S}$ is the standard
$\cO_S$-Lie algebroid on $X$, the sheaf ${\ti \Lambda}_L'$ is a
sheaf of Azumaya $\cO_{\VV (L')}$-algebras. In this case we have a
canonical splitting
$${\ti F}_{X/S}^*{\ti\Lambda}_L'\simeq \cEnd _{\cO_{\VV (F^*_{X/S}L')}}
{\ti\Lambda}_{L}.$$

\subsection{Relation with groupoid schemes}

This subsection contains a quick tour on relation between Lie
algebroids and groupoid schemes of height $\le 1$. This is
analogous to the well-known relation between  restricted Lie
algebras and group schemes of height $\le 1.$

\medskip

Let us recall that a \emph{groupoid} is a small category in which every morphism is
an isomorphism. Let $X$ and $R$ be $S$-schemes.
An \emph{$S$-groupoid scheme} $G$ is a quintuple of $S$-maps $s,
t: R\to X$ (``source and target objects''),  $c:R\times
_{(s,t)}R\to R$ (``composition''), $e:X\to R$ (``identity map'')
and $i: R\to R$ (``inverse map'') such that for every $S$-scheme
$T$ the quintuple $s(T)$, $t(T)$, $c(T)$, $e(T)$ and $i(T)$
defines in a functorial way a groupoid  with morphisms $R(T)$ and
objects $X(T)$.

For an $S$-groupoid scheme $G$ we denote by $\cJ$ the kernel of
$s_*\cO_R\to \cO_X$. We say that $G$ is \emph{infinitesimal} if
$s$ is an affine homeomorphism and $\cJ$ is a nilpotent ideal. An
infinitesimal $S$-groupoid scheme is of \emph{height $\le 1$} if
$(s,t):R\to X\times _SX$ factors through the first Frobenius
neighbourhood of the diagonal (i.e., through $X\times
_{X^{(1/S)}}X$). An $S$-groupoid scheme is called \emph{finite}
(\emph{flat}) if $s$ is finite (respectively, flat).

If $X$ is smooth over a perfect field $k$ then restricted $k$-Lie
subalgebras $L$ of the standard $k$-Lie algebroid $T_{X/k}$ such
that $T_{X/k}/L$ is locally free are in bijection with finite flat
height $1$ morphisms $X\to Y$ (see \cite[Proposition 2.4]{Ek}).
Note that a sheaf of restricted $k$-Lie subalgebras of $T_{X/k}$
is automatically a restricted $k$-Lie subalgebroid of $T_{X/k}$.
So the following proposition generalizes the above fact (and it
corrects \cite[Proposition 2.3]{Ek}):

\begin{Proposition}
Let $X/S$ be a smooth morphism. Assume that for every point $x\in
X $ the set $t(s^{-1}(x))$ is contained in an affine open subset
of $X$. Then there exists an equivalence of categories between the
category of finite flat $S$-groupoid schemes of height $\le 1$
with $X/S$ as a scheme of objects and with locally free ``conormal
sheaf'' $\cJ/\cJ ^2$ and the category of smooth restricted
$\cO_S$-Lie algebroids on $X/S$.
\end{Proposition}

\begin{proof}
We sketch the proof leaving details to the reader.

If $G$ is a finite, flat, infinitesimal $S$-groupoid scheme then
we define $L$ as the Lie algebra of this groupoid, i.e., the dual
of  $\cJ/\cJ ^2$. It has a natural structure of a sheaf of
restricted $\cO_S$-Lie algebras. Since $G$ has height $\le 1$, $L$
is equipped with the anchor map.

In the other direction, to a smooth restricted $\cO_S$-Lie
algebroid $L$ on $X/S$ we associate  $\Lambda_L^{[p]}$ which comes
with a canonical homomorphism of $\cO_S$-algebras
$\Lambda_L^{[p]}\to \Lambda_{T_{X/S}}^{[p]}$. But
$\Lambda_{T_{X/S}}^{[p]}$ is an $\cO_S$-subalgebra of the sheaf of
rings of ``true'' differential operators and the ``morphisms'' $R$
of the groupoid scheme can be defined as the spectrum of the dual
of $\Lambda_L^{[p]}$.
\end{proof}

\subsection{Modules over restricted Lie algebroids}

If $E$ is a module over a restricted $\cO_S$-Lie algebroid $L$
then $\nabla :L\to \cEnd_{\cO_S} E$ leads to a morphism
$$\psi: L\to \cEnd _{\cO_S}E$$
defined by sending $x$ to $(\nabla (x))^p-\nabla (x^{[p]})$ for
$x\in L$.

Let us set $\alpha _x^0(f)=f$ and $(\nabla(x))^0(e)=e$.
Using Leibniz' rule one can easily see that
$$(\nabla(x))^m(fe)=\sum _{i=0}^m{m \choose i} \alpha_x^i(f)(\nabla (x))^{m-i}(e)$$
for any sections $f\in \cO _X(U)$, $x\in L (U)$ and $e\in E(U)$
and any open subset $U\subset X$.
In particular, we have
$$(\nabla(x))^p(fe)=\alpha_x^p(f)e + f(\nabla(x))^p(e).$$
Since
$$\nabla(x^{[p]}) (fe)=\alpha_{x^{[p]}}(f)e + f\nabla(x^{[p]})(e)$$
and $\alpha_{x^{[p]}}=\alpha_x^p$ we see that for any $x\in L$ the
image $\psi (x)$ is $\cO_X$-linear. So we can consider $\psi$ as
the mapping $\psi: L\to \cEnd _{\cO_X}E$. This mapping is called
the \emph{$p$-curvature morphism} of the $L$-module $E$. The
following lemma generalizes \cite[Proposition 5.2]{Ka1}:

\begin{Lemma} \label{p-cur}
The $p$-curvature morphism $\psi : L\to \cEnd _{\cO_X}E$ is
$F_X^*$-linear and its image commutes with the image of
$\nabla$ in $\cEnd _{\cO_S}E$.
\end{Lemma}

\begin{proof}
By Lemma \ref{p-curvature-for-Lie-algebras} we know that $\psi$ is
additive and its image commutes with the image of $\nabla$.  So it
is sufficient to check that
$$\psi (fx)=f^p\psi (x)$$
for all local sections $f\in \cO_X$ and $x\in L$. Applying
Hochschild's identity to elements $f$ and $\nabla (x)$ in $\cEnd
_{\cO_S}E$ we obtain
$$(\nabla (fx))^{p}=f^p \nabla(x)^{p}+(\ad (f\nabla(x)))^{p-1}(f)\nabla (x)
=f^p \nabla(x)^{p}+\alpha ^{p-1}_{fx}(f) \nabla (x).$$ From the
definition of a restricted $\cO_S$-Lie algebroid and
$\cO_X$-linearity of $\nabla: L\to \cEnd _{\cO_S} E$ we have
$$\nabla ((fx)^{[p]})=f^p \nabla(x^{[p]})+\alpha ^{p-1}_{fx}(f) \nabla (x).$$
Subtracting these equalities we get the required identity.
\end{proof}

\medskip
By the above lemma $\psi$ defines an $\cO_X$-linear map $L\to
F_{X, *}\cEnd _{\cO_X}E$ and hence the adjoint $\cO_X$-linear map
$$\psi_L: F_{X}^*L\to \cEnd _{\cO_X}E,$$
which will also be called the $p$-curvature morphism. Note that
$\psi _L$ makes $E$ into an $F_X^* L$-coHiggs sheaf (integrability
of the $F_X^* L$-coHiggs field follows from the lemma). Another
way of seeing it is that if $E$ is a $\Lambda_L$-module then by
Lemma \ref{BMR-lemma} it has a structure of
$S^{\bullet}(F_X^*L)$-module given by the $p$-curvature $\psi_L$.

\medskip

\begin{Example} \label{example:trivial-Lie-algebroid}
Let $L$ be a smooth trivial restricted $\cO_S$-Lie algebroid on
$X$. Then giving an $L$-module is equivalent to giving
$S^{\bullet}L$-module structure on $E$. In this case the
$p$-curvature morphism $\psi _L: F_X^* L\to \cEnd _{\cO_X}E$ is
obtained by composing the canonical inclusion $F_X^* L\to S^{p}L$
with the action map $S^p L\to \cEnd _{\cO_X}E$.
\end{Example}

\medskip

\begin{Example} \label{example:lambda-connection}
Let $X$ be a smooth $S$-scheme and let us fix $\lambda \in
H^0(\cO_S)$. Let us denote by  $T_{X/S}^{\lambda}$ the restricted
$\cO_S$-Lie algebroid structure on $T_{X/S}$ with Lie bracket
$[\cdot, \cdot ]_{T_{X/S}^{\lambda}}=\lambda \cdot [\cdot, \cdot
]_{T_{X/S}}$, anchor map $\alpha$ given by multiplication by
$\lambda$ and the $p$-th power operation given by
$$x^{[p]}_{T_{X/S}^{\lambda}}=\lambda^{p-1}\cdot
 x^{[p]}_{T_{X/S}} $$
for $x\in T_{X/S}$. The apparently strange formula for the $p$-th
power operation comes from the requirement $$\alpha
(x^{[p]}_{T_{X/S}^{\lambda}})=\lambda \cdot
x^{[p]}_{T_{X/S}^{\lambda}}= \left( \alpha (x)\right)^{[p]}
=\lambda^p\cdot x^{[p]}.$$

Giving a $T_{X/S}^{\lambda}$-module is equivalent to giving a
coherent $\cO_X$-module $E$ with an integrable
$\lambda$-connection $\nabla : E\to E\otimes
_{\cO_X}\Omega_{X/S}$. In this case the above defined
$p$-curvature of the $T_{X/S}^{\lambda}$-module gives a more
conceptual approach to the $p$-curvature of an $\cO_X$-module with
$\lambda$-connection $(E,\nabla)$ defined in \cite[Definition
3.1]{LP}.
\end{Example}

\medskip

\begin{Remark}
If $\nabla_1$ and $\nabla_2$ are two $L$-module structures on $E$
then $\varphi=\nabla_1-\nabla_2: L\to \cEnd _{\cO_S}E$ is
$\cO_X$-linear and its image lies in $\cEnd _{\cO_X}E$. In
particular, if the $p$-curvatures $\psi _L(\nabla _1)$ and $\psi
_L (\nabla_2)$ are equal then $\varphi$ is zero on the kernel of
$\Lambda_L\to \Lambda_L^{[p]}$ and hence it induces the
homomorphism $\Lambda_L^{[p]}\to \cEnd_{\cO_X}E$ of
$\cO_X$-algebras.
\end{Remark}

\medskip

\begin{Definition}
We say that the $p$-curvature of $(E, \nabla)$ is \emph{nilpotent
of level less than $l$} if  $(E, \nabla)$ satisfies one of the
following equivalent conditions:
\begin{enumerate}
\item There exists a filtration $M^m=0\subset M^{m-1}\subset ...\subset M^0=(E, \nabla)$ of length
$m\le l$ such that the associated graded $L$-module has
$p$-curvature $0$.
\item For any open subset $U\subset X$ and any collection $\{ x_1,..., x_l\}$ of
sections of $L(U)$ we have $\psi _L(x_1)...\psi_L(x_l)=0$.
\end{enumerate}
We say that he $p$-curvature of $(E, \nabla)$ is \emph{nilpotent
of level $l$} if it is nilpotent of level less than $(l+1)$ but
not nilpotent of level less than $l$ (for $l=0$ we require simply
that the $p$-curvature is nilpotent of level less than $1$).
\end{Definition}

\subsection{Deformation of Hitchin's morphism for restricted Lie
algebroids}

This subsection contains a partial generalization of the results
of Laszlo and Pauly \cite{LP} to higher dimensions. Note that in
general, the direct analogue of their \cite[Proposition 3.2]{LP}
is not expected to be true.

\medskip

Let $S$ be a noetherian scheme of characteristic $p$ and let $X\to
S$ be a flat, projective family of $d$-dimensional varieties
satisfying Serre's condition $(S_2)$. Let $L$ be a smooth
restricted $\cO_S$-Lie algebroid on $X$. Let us fix a polynomial
$P$ and a relatively ample line bundle on $X/S$. We define the
moduli stack as a lax functor between $2$-categories by
$$\begin{array}{cccc}
\cM^L(X/S, P): &(\Sch /S)&\to& (\hbox{groupoids})\\
&T&\to&\cM(T),
\end{array}$$
where $\cM(T)$ is the category whose objects are $T$-flat families
of pure $d$-dimensional $L$-modules with Hilbert polynomial $P$ on
the fibres of $X_T\to T$, and whose morphisms are isomorphisms of
coherent sheaves. One can prove that $\cM^L(X/S, P)$ is an
algebraic stack for the fppf topology on $(\Sch /S)$. If $M$ is a
coherent $\cO_X$-module considered as an $\cO_S$-Lie algebroid on
$X$ with the trivial structure, then the corresponding moduli
stack is denoted by $\cM^M_{\rm Dol}(X/S, P)$.

The $p$-curvature defines a morphism of stacks
$$\begin{array}{cccc}
\Psi_L: & \cM^L(X/S, P)&\to&  \cM^{F_X^*L}_{\rm Dol}(X/S, P)\\
&(E, \nabla)&\to&(E,\psi (\nabla)).
\end{array}
$$

Let us consider the deformation $L^R$ of $L$ over an affine line
$\AA^1$ over ${\FF_p}$ (see Subsection
\ref{deformation-over-line}). For simplicity of notation, in the
following  we skip writing ${\FF_p}$. $L^R$ has a natural
structure of a smooth restricted $\cO_{S\times \AA^1}$-Lie
algebroid on $X\times \AA^1$ with the $p$-th power operation given
by $\cdot^{[p]}_{L^R}=p_1^*\left(\cdot^{[p]}_{L}\right)\otimes
t^{p-1}$. We can treat $L^R$ as a family of restricted $\cO_S$-Lie
algebroids on $X$ parameterized by $\AA ^1$. For example, if $X/S$
is smooth and we fix $\lambda \in H^0(\cO_S)=\Hom (S, \AA^1)$ then
for $L=T_{X/S}$ with the standard restricted $\cO_S$-Lie algebroid
structure, the pull-back of $L^R$ along $(\id _S, \lambda ) : S\to
S\times \AA ^1$ gives $T_{X/S}^{\lambda}$ from Example
\ref{example:lambda-connection}.

We have a commutative diagram
$$ \footnotesize{\xymatrix{ \cM^{L^R}(X\times \AA^1 /S\times \AA^1 , P)
\ar[r]^{\Psi_{L^R}} & \cM ^{F_{X\times \AA^1}^*L^R}_{\rm
Dol}(X\times \AA^1 /S\times \AA^1, P )\ar[r]^{H_{F_{X\times
\AA^1}^*L^R}}& \VV ^{
F_{X\times \AA^1}^*L^R}(X\times \AA^1 /S\times \AA^1 ,r)\\
\cM ^L_{\rm Dol}(X/S, P)\ar[r]^{H_L}\ar[u] & \VV ^{L}(X/S,
r)\ar[r]& \VV^{ F_X^*L}(X/S,r),\ar[u]\\
}}$$ where the vertical arrows are induced by the base change via
the zero section $0:S\to S\times \AA^1$ and $ \VV ^{L}(X/S,r) \to
\VV ^{ F_X^*L}(X/S,r)$ is the canonical morphism induced by the
absolute Frobenius on $X$. Roughly speaking, this diagram says
that the $p$-curvature morphism $\Psi_L$ deforms to the $p$-th
power of the Hitchin morphism.

\medskip

Let $\Nilp ^L(X/S,P)$ be the substack of $ \cM ^L(X/S, P)$ of
$L$-modules with nilpotent $p$-curvature. By definition $\ti \Psi$
maps  $\Nilp ^L(X/S,P)$ into $\{0\}\times \AA^1=\AA^1$ and the
corresponding map will be still denoted by $\ti \Psi$. The stacks
$ \cM^L(X/S, P)$ and $\Nilp ^L(X/S,P)$ contain open substacks $
\cM^{L,ss}(X/S, P)$ and $\Nilp ^{L,ss}(X/S,P)$ parametrizing slope
semi-stable objects (openness of semistability is a standard
exercise left to the reader). By boundedness theorem (see
\cite{La2}) these substacks are of finite type. Theorem
\ref{slope-Langton} implies that the morphisms $ {\ti
\Psi}^{ss}:\cM^{L^R,ss}(X/S, P)\to  \VV ^{ F_X^*L}(X/S,r) \times
\AA^1$ and $ \Nilp^{L,ss}(X/S, P)\to \AA^1$ are universally
closed.

\medskip

Let $\Nilp _{l}^{L,ss}(X/S,P)$ be the substack of  $\Nilp
^{L,ss}(X/S, P)$ parametrizing objects with nilpotent
$p$-curvature of level $<l$. Note that it is a closed substack,
since nilpotence of level $<l$ is a closed condition. Therefore $
\Nilp_{l}^{L,ss}(X/S, P)\to \AA^1$ is universally closed (see
\cite[Proposition 5.1]{LP} for a special case of this assertion).

Let us note that the fiber of $ \Nilp_{1}^{L,ss}(X/S, P)\to \AA^1$
over $0$ is equal to the moduli stack of semistable $L$-coHiggs
sheaves $(E, \theta)$ with vanishing $p$-curvature (see Example
\ref{example:trivial-Lie-algebroid}). In particular, \cite[Remark
5.1]{LP} is false.

On smooth projective curves of genus $g\ge 2$ the proof of
\cite[Lemma 5.1]{LP} shows that a vector bundle with a
$\lambda$-connection of level less than $l$ can be extended to a
Higgs bundle with the Higgs field $\theta$ satisfying $\theta
^l=0$. In particular, for $l=1$ we get the zero Higgs field.

So one could hope that in this case, e.g., if $
{\widetilde{\Nilp}}_{1}^{L,ss}(X/S, P)\to \AA^1$ is the open
substack of $ \Nilp_{1}^{L,ss}(X/S, P)\to \AA^1$, which over $0$
is the moduli substack of semistable sheaves then
${\widetilde{\Nilp}}_{1}^{L,ss}(X/S, P)\to \AA^1$ is also
universally closed as suggested by \cite[Remark 5.1]{LP}. However,
this expectation is false. In case of a smooth projective  curve
$X$ of genus $g\ge 2$ there exists a semistable bundle $E$ whose
Frobenius pull back $F_X^*E$ is not semistable. But $F_X^*E$
carries a canonical connection $\nabla _{\can}$ and $(F_X^*E,
\nabla _{\can})$ is semistable. After pulling back via $X_K\to X$,
where $K=k((t))$, and twisting by $t$, this provides a semistable
vector bundle with a $t$-connection on $X_{K}$ which cannot be
extended to a semistable family on $X_{k[[t]]}$ so that the Higgs
field at the special fibre vanishes. Otherwise, we would get a
contradiction with openness of the usual semistability of vector
bundles.

\section{Deformations of semistable sheaves and
the Lan-Sheng-Zuo conjecture}

\subsection{Langton's theorems}

Let $R$ be a discrete valuation ring with maximal ideal $m$
generated by $\pi \in R$. Let $K$ be the quotient field of $R$ and
let us assume that the residue field $k=R/m$ is algebraically
closed.

Let $X\to S=\Spec R$ be a smooth projective morphism and let $L$
be a smooth $\cO_S$-Lie algebroid on $X$. Let us fix a collection
$(D_0, D_1,\dots ,D_{n-1})$ of $n$ relatively nef divisors on
$X/S$ such that $D_0=D_1$. In the following stability of sheaves
on the fibers of $X\to S$ is considered with respect to this fixed
collection.

The following theorem generalizes well known Langton's theorem
\cite[Theorem 2)]{Lt}. We recall the proof  as it is not available
in the generality that we need. The notation introduced in this
proof will be also used in proof of Theorem
\ref{existence-of-gr-ss-Griffiths-transverse-filtration}.

\begin{Theorem}\label{slope-Langton}
Let $F$ be an $R$-flat $\cO_X$-coherent $L$-module of relative
pure dimension $n$ such that the $L_K$-module $F_K=F\otimes _RK$
is slope semistable. Then there exists an $L$-submodule $E\subset
F$ such that $E_K=F_K$ and $E_k$ is a slope semistable
$L_k$-module on $X_k$.
\end{Theorem}

\begin{proof}
First let us note that we can assume that $F_k$ is torsion free as
an $\cO_{X_k}$-module (this follows, e.g., from \cite[Proposition
4.4.2]{HL} or can be proven using a similar method as below). We
use without warning the fact that for an $R$-flat $F$ the degrees
of $F_K$ and $F_k$ with respect to $(D_1,\dots ,D_{n-1})$
coincide. This follows from the fact that $F$ has a finite locally
free resolution on $X$ and intersection products are compatible
with specialization (see \cite[Expose X, Appendice]{SGA6}).

Let us set $F^0:=F$. If $F_k^0$ is not slope semistable then we take
the maximal destabilizing $L$-submodule $B^0$ in $F_k^0$ and denote
by $F^1$ the kernel of the composition $F^0\to F_k^0\to G^0:=F_k^0/B^0$. If
$F^1_k$ is semistable then we get the required submodule of $F$.
Otherwise, we repeat the same procedure for $F^1$. In this way we
construct a sequence of $L$-modules $F=F^0\supset F^1\supset
F^2\supset ...$ and the main point of the proof is to show that this
process cannot continue indefinitely.

Let us assume otherwise. First, let us note that we have short exact sequences
$$0\to G^n \to F^{n+1}_k\to B^n \to 0,$$
where $G^n=  F^n_k/B^n$. Let $C^n$ be the kernel of the
composition $B^{n+1}\to F^{n+1}_k\to B^n$.

If $C_n=0$ then $B_{n+1}\subset B_n$ and hence $\mu (B^{n+1})\le
\mu (B^{n})$. If $C^n\ne 0$ then
$$\mu (C^n)\le \mu _{\max}(G^n)< \mu (B^{n}),$$
where the first inequality comes from the fact that $C^n\subset
G^n$ and the second one follows from the fact that $B^n\subset
F^n_k$ is the maximal destabilizing subsheaf and $G^n=F^n_k/B^n$.

We claim that $\mu (B^{n+1})<\mu (B^{n})$. If $\mu (C^n)\ge \mu
(B^{n+1})$ then this inequality follows from the above inequality.
If $\mu (C^n)< \mu (B^{n+1})$ then $\mu (B^{n+1})<\mu
(B^{n+1}/C^n)$. But $B^{n+1}/C^n$ is isomorphic to a subsheaf of
$B^n$ and $B^n$ is semistable, so in this case we also have $\mu
(B^{n+1})<\mu (B^{n})$.

Therefore the sequence $\{\mu (B^{n})\}$ is non-increasing. But
$\mu (B^{n+1})<\mu (B^{n})$ is possible for only finitely many $n$
since  $r! \mu (B^n)\in \ZZ$ are bounded below by $r!\mu (F_k)$.
Therefore for all large $n$ we have $C^n=0$, i.e., we have
inclusions $B^n\supset B^{n+1}\supset ...$ and $G^n\subset
G^{n+1}\subset...$. For sufficiently large $n$ these sequences
consist of torsion free sheaves with the same slope, so they must
stabilize to $B$ and $G$, respectively. Then $F^{n}_k=B\oplus G$
for $n\gg 0$. Set $\hat R:={\displaystyle \lim_{\leftarrow}}\,
R/\pi^nR$ and let $\hat K$ be the quotient field of $\hat R$. Note
that $F/F^n$ is $R/\pi^n$-flat and as $\cO_{X_k}$-module has a
filtration with quotients isomorphic to $G$. Then $\hat
Q:={\displaystyle \lim_{\leftarrow}} \,  F/F^n$ is a destabilizing
quotient of $F_{\hat K}$. But the Harder--Narasimhan filtration is
stable under base field extension and therefore $F_K$ is also
unstable, contradicting our assumption.
\end{proof}

\medskip
Our exposition of proof of Langton's theorem is based on \cite{HL}
with some small changes (one of the inequalities in proof of
\cite[Theorem 2.B.1]{HL} is false and we need to give a slightly
different argument).

\medskip

Note that in the above theorem we allow the case when all $D_i$'s are zero.
In this case we claim that there exists an $L$-submodule $E\subset F$
such that $E_K=F_K$ and $E_k$ is torsion free as $\cO_{X_k}$-module (by definition
slope semistable sheaves are torsion free!).

\medskip
Let us recall that every slope semistable $L$-module $E$ has a
Jordan--H\"older filtration $E_0=0\subset E_1\subset ...\subset
E_m=E$ by $L$-submodules such that the associated graded sheaf
$\Gr (E)=\oplus E_i/E_{i-1}$ is \emph{slope polystable}, i.e., a
direct sum of slope stable (torsion free) $L$-modules of the same
slope.

The following theorem is motivated by theory of moduli spaces and
it generalizes \cite[Theorem 1)]{Lt}.

\begin{Theorem}\label{slope-Langton2}
Assume that the collection $(D_0, D_1,\dots ,D_{n-1})$ consists of
relatively ample divisors. Let $F$ be an $R$-flat $\cO_X$-coherent
$L$-module of relative pure dimension $n$ such that the
$L_K$-module $F_K=F\otimes _RK$ is slope semistable. Let $E_1$ and
$E_2$ be $L$-submodules of $F$ such that $(E_1)_K=(E_2)_K=F_K$,
$(E_1)_k$ and $(E_2)_k$ are slope semistable. Then the
reflexivizations of the associated graded slope polystable sheaves
$\Gr ((E_1)_k)$ and $\Gr ((E_2)_k)$ are isomorphic. Moreover, if
at least one of $(E_1)_k$ and $(E_2)_k$ is slope stable then there
exists an integer $n$ such that $E_1=\pi^nE_2$.
\end{Theorem}

\begin{proof}
We prove only the second part, leaving proof of the first one to
the reader. Assume that $(E_1)_k$ is slope stable. Consider the
discrete valuation ring $\cO_{X, \eta}$, where $\eta$ is the
generic point of $X_k$. Multiplying $E_1$ by some power of $\pi$,
we can assume that $E_1\otimes _{\cO_X} \cO_{X, \eta}\subset
E_2\otimes _{\cO_X}\cO_{X, \eta}$ and the induced map $E_1\otimes
k({\eta})\to E_2\otimes k(\eta)$ is non-zero. But $E_1$ and $E_2$
are torsion free so $E_1\subset E_2$. Since $(E_1)_k$ is slope
stable the non-zero map $(E_1)_k\to (E_2)_k$ between slope
semistable sheaves of the same slope must be an inclusion. Since
the Hilbert polynomials of $(E_1)_k$ and $(E_2)_k$ coincide (from
flatness of $E_1$ and $E_2$), it must be an isomorphism.
\end{proof}

\medskip

Let $Y$ be a projective scheme over a field $k$ and let $L_Y$ be a
$k$-Lie algebroid  on $Y$. Let us fix an ample line bundle
$\cO_Y(1)$ on $Y$. Let $\Coh ^L_d(Y)$ be the full subcategory of
the category of $L$-modules which are coherent as $\cO_Y$-modules
and whose objects are sheaves supported in dimension $\le d$. Then
we can consider the quotient category $\Coh ^L_{d,d'}(Y):=\Coh
^L_{d}(Y)/\Coh ^L_{d'-1}(Y)$. For any object of $\Coh
^L_{d,d'}(Y)$ one can define its Hilbert polynomial which can be
used to define notion of (semi)stability in this category.

\medskip
We can generalize Langton's theorem to singular schemes at the
cost of dealing with only one ample polarization. In this case
compatibility of intersection product with specialization follows
from computation of the Hilbert polynomial. One can also
generalize Theorem  \ref{slope-Langton} so that it works for other
kinds of stability as defined above.

Let $X\to \Spec R$ be a projective morphism with relatively ample 
line bundle $\cO_X(1)$ and let $L$ be a smooth $\cO_S$-Lie algebroid on $X$. 
The following Langton's type theorem  generalizes \cite[Theorem 10.1]{Si5} 
and \cite[Theorem 2.B.1]{HL}:

\begin{Theorem}\label{Langton}
Let $F$ be an $R$-flat $\cO_X$-coherent $L$-module of relative
dimension $d$. Assume that the $L_K$-module $F_K=F\otimes _RK$
is pure of dimension $d$ and semistable in $\Coh ^L_{d,d'}(X_K)$ for some $d'<d$. 
Then there exists an $L$-submodule $E\subset F$ such that $E_K=F_K$ and $E_k$
is semistable in $\Coh ^L_{d,d'}(X_k)$.
\end{Theorem}

\begin{proof}
The proof is almost the same as the proof of \cite[Theorem
2.B.1]{HL}. However, there are a few small problems that we meet in the proof.
The first one is that we need to define reflexive hulls of sheaves
on the special fiber $X_k$. This can be done by embedding $X$ into
a fixed smooth $R$-scheme (e.g., use some multiple of the polarization 
$\cO_X(1)$ to embedd $X$ into some projective space over $R$). 

The second problem is the same as before: one of the inequalities in proof of  
\cite[Theorem 2.B.1]{HL} is false and we need to use a slightly different argument similar to the one used in proof of Theorem \ref{slope-Langton}. 
We sketch the necessary changes using the notation of proof of 
\cite[Theorem 2.B.1]{HL}. 
If $C^n\ne 0$ then we only have
$$p(C^n)\le p_{\max}(G^n)< p(B^n) \mod \, \QQ[T]_{\delta -1}.$$ 
Hence if $p_{d, \delta}(C^n)\ge p_{d, \delta}(B^{n+1})$ then 
$p_{d, \delta}(B^{n+1})<p_{d, \delta}(B^{n})$.  If 
$p_{d, \delta}(C^n)< p_{d, \delta}(B^{n+1})$  then we have
$p_{d, \delta}(B^{n+1})< p_{d, \delta}(B^{n+1}/C^n)\le p_{d, \delta}(B^{n}).$
This proves that if $C^n\ne 0$ then we always have $p_{d, \delta}(B^{n+1})<p_{d, \delta}(B^{n})$ as needed in the argument.

The last problem is the use of Quot-schemes in \cite{HL}, which do not exist as
projective schemes in our situation. This can be solved as in proof of Theorem
\ref{slope-Langton}.
\end{proof}

\begin{Theorem}\label{Langton2}
Let $F$ be an $R$-flat $\cO_X$-coherent $L$-module of relative
pure dimension $d$ such that the $L_K$-module $F_K=F\otimes _RK$
is semistable in $\Coh ^L_{d,d'}(X_K)$ for some $d'<d$.
Let  $E_1$ and $E_2$ be $L$-submodules of  $F$ such that $(E_1)_K=(E_2)_K=F_K$,
$(E_1)_k$ and $(E_2)_k$ are semistable in $\Coh ^L_{d,d'}(X_k)$
and at least one of them is stable. Then there exists an integer $n$ such that
$E_1=\pi^nE_2$ in $\Coh ^L_{d,d'}(X_K)$.
\end{Theorem}

\subsection{Semistable filtrations on sheaves with connection}

Let $L$ be a smooth Lie algebroid on a normal projective variety
$X$ defined over an algebraically closed field $k$. Let us
consider a torsion free coherent $\cO _X$-module $E$ with an
integrable $d_{\Omega_L}$-connection $\nabla $ (i.e., an
$L$-module whose underlying sheaf is coherent and torsion free as
an $\cO_X$-module). We say that a filtration $E=N^0\supset
N^1\supset ...\supset N^m=0$ satisfies \emph{Griffiths
transversality} if $\nabla (N^i)\subset N^{i-1}\otimes
_{\cO_X}\Omega_L$ and the quotients $N^i/N^{i+1}$ are torsion
free. For every such filtration the associated graded object $\Gr
_N(E):=\bigoplus _i N^i/N^{i+1}$ carries a canonical
$\Omega_L$-Higgs field $\theta$ defined by $\nabla$. Note that
$(\Gr_N(E), \theta)$ is a system of $L$-Hodge sheaves. A
convenient way of looking at this is by means of the Rees
construction. More precisely, if $N^{\bullet}$ is a Griffiths
transverse filtration on $(E, \nabla)$ then we can consider the
subsheaf
$$\xi (E, N^{\bullet}):=\sum t^{-i}N^i\otimes \cO_{X\times \AA ^1}\subset p_X^*E$$
on $X\times \AA ^1$. By Griffiths transversality of the filtration
$N^{\bullet}$ the connection $t\nabla$ on $\xi (E,
N^{\bullet})|_{X\times \GG_m}$ extends to a
$t$-$d_{\Omega_L}$-connection on $X\times \AA ^1$ (i.e., we get an
$L^R$-module on $X\times \AA^1$). In the limit as $t\to 0$ we get
exactly the above described system of $L$-Hodge sheaves
$(\Gr_N(E), \theta)$.

\medskip
In the remainder of this section to define semistability we use
a fixed collection $(D_0,D_1,\dots ,D_{n-1})$ of nef divisors such that $D_0=D_1$.

After Simpson \cite{Si5} we say that a Griffiths transverse
filtration $N^{\bullet}$ on $(E, \nabla)$ is \emph{slope
gr-semistable} if the associated $\Omega_L$-Higgs sheaf
$(\Gr_N(E), \theta)$ is slope semi\-sta\-ble. A \emph{partial
$L$-oper} is a triple $(E, \nabla, N^{\bullet})$ consisting of a
torsion free coherent $\cO _X$-module $E$ with an integrable
$d_{\Omega_L}$-connection $\nabla$ and a Griffiths transverse
filtration $N^{\bullet}$ which is slope gr-semistable.

\begin{Theorem} \label{existence-of-gr-ss-Griffiths-transverse-filtration}
If $(E, \nabla)$ is slope semistable then there exists a
canonically defined slope gr-semistable Griffiths transverse
filtration $N^{\bullet}$ on $(E, \nabla)$ providing it with a
partial $L$-oper structure. This filtration is preserved by the
automorphisms of $(E,\nabla)$.
\end{Theorem}

\begin{proof}
Let $R$ be a localization of $\AA ^1$ at $0$ and let $\ti L$ be
the smooth Lie algebroid on $X_R=X\times _k \Spec R$ obtained by
restricting of $L^R$ from $X\times _k\AA^1$. Consider the trivial
filtration $E=N^0\supset N^1=0$. It satisfies Griffiths
tranversality so we can associate to it via the  Rees construction
and restricting to $X_R$, an $R$-flat $\cO_{X_R}$-coherent $\ti
L$-module $F^0=F$ (in fact $F=(p_X^*E, \pi \nabla )$).

Now suppose that we have defined an $\ti L$-submodule $F^n\subset
F$ coming by restriction from the Rees construction associated to
a Griffiths transverse filtration $N_n^{\bullet}$ of $E$. If the
associated $\Omega_L$-Higgs sheaf $F^n_k=(\Gr_{N_n}(E), \theta_n)$
is semistable then we get the required filtration. Otherwise, we
consider its maximal destabilizing $\Omega_L$-Higgs subsheaf
$B^n$. But $(\Gr_{N_n}(E), \theta _n)$ is a system of
$\Omega_L$-Hodge sheaves, so by Corollary \ref{Higgs=Hodge} $B^n$
is also a system of $\Omega_L$-Hodge sheaves. Let us write
$B^n=\bigoplus B^n_m$, where $B^n_m\subset
\Gr_{N_n}^m(E)=N_n^m/N_n^{m+1}$. Then we can define a new
Griffiths transverse filtration $N_{n+1}^{\bullet}$ on $E$ by
setting
$$N_{n+1}^{m}:=\ker \left(E\to \frac{E/N_n^m}{B^n_{m-1}} \right).$$
Let $F^{n+1}$ denote the restriction to $X_R$ of the $L^R$-module
associated by the Rees construction to $N_{n+1}^{\bullet}$. We
need to prove that this procedure cannot continue indefinitely. To
show it, it is sufficient to check that we follow the same
procedure as the one described in the proof of Theorem
\ref{Langton}.

By construction $\pi F^n\subset F^{n+1}\subset F^n$ and in
particular $F^{n+1}_K=F^n_K$. On the other hand, on the special
fiber of $X_R\to \Spec R$ we have a short exact sequence
$$0\to F^n_k/B^n\to F^{n+1}_k=({\Gr}_{N_{n+1}}(E), \theta_{n+1}) \to B^n \to 0$$
coming from the definition of the filtration $N_{n+1}^{\bullet}$.
This shows that $\pi F^{n}$ is the kernel of the composition
$F^{n+1}\to F^{n+1}_k\to B^n$. But then $F^{n+1}$ is the kernel of
the composition $F^n\to F^n_k\to F^n_k/B^n$. Now the proof of
Theorem \ref{Langton} shows that this procedure must finish.

Since the Harder--Narasimhan filtration is canonically defined, the
above described procedure is also canonical and the obtained filtration is
preserved by the automorphisms of $(E, \nabla)$.
\end{proof}

\medskip
In the following the canonical filtration $N^{\bullet}$
from Theorem \ref{existence-of-gr-ss-Griffiths-transverse-filtration}
will be called \emph{Simpson's filtration} of $(E, \nabla)$ and denoted by $N^{\bullet}_S$.
The reason is
that apart from many spectacular results due to Simpson in
non-abelian Hodge theory, the construction of the filtration
described in the proof of the above theorem was done by Simpson in
\cite[Section 3]{Si5} for the usual Higgs bundles on complex
projective curves. However, our proof of the fact that the
procedure stops is different.

Theorem \ref{existence-of-gr-ss-Griffiths-transverse-filtration}
generalizes \cite[Theorem 2.5]{Si5} to higher dimensions as asked
for at the end of  \cite[Section 3]{Si5}. Indeed, in the
characteristic zero case every vector bundle with an integrable
connection has vanishing Chern classes. In particular, any
saturated subsheaf of such a vector bundle which is preserved by
the connection (is locally free and) has vanishing Chern classes.
So any vector bundle with an integrable connection is slope
semistable (with respect to an arbitrary polarization). This
argument fails in the logarithmic case which shows that the above
theorem is a correct analogue in this case.

\medskip

Note that there can be many slope gr-semistable filtrations
providing $(E, \nabla)$ with a partial $L$-oper structure. This
depends on the choice of the Griffiths transverse filtration at
the beginning of our procedure (in the proof of Theorem
\ref{existence-of-gr-ss-Griffiths-transverse-filtration} we used
the canonical choice). In general, all the obtained filtrations
are related as described by the following corollary
which follows from Theorem \ref{slope-Langton2}:

\begin{Corollary}\label{YY}
If $N^{\bullet}$ and $M^{\bullet}$ are two slope gr-semistable
Griffiths transverse filtrations on $(E, \nabla)$ then the
reflexivizations of the associated-graded slope poly\-stable
$\Omega_L$-Higgs sheaves obtained from their Jordan--H\"older
filtrations are isomorphic. In particular, if the associated
$\Omega_L$-Higgs sheaf is slope stable then $(E, \nabla)$ carries
a unique gr-semistable Griffiths transverse filtration.
\end{Corollary}

The above corollary generalizes \cite[Corollary 4.2]{Si5}. Note
that Simpson's proof does not work so easily in our situation as
in higher dimensions we do not have appropriate moduli spaces at
our disposal.

Let us also note that any slope gr-semistable filtration can be
refined so that the associated graded $\Omega_L$-Higgs sheaf is
slope polystable (in which case its reflexivization is uniquely
determined by $(E, \nabla)$ up to an isomorphism).

\medskip

As an immediate application of Theorem \ref{existence-of-gr-ss-Griffiths-transverse-filtration}
we also get the following interesting corollary:

\begin{Corollary}\label{deforming-Higgs-to-Hodge}
Let $L$ be a smooth trivial Lie algebroid. Let $(E, \theta)$ be a torsion free, slope semistable
$\Omega_L$-Higgs sheaf on $X$. Then we can deform it to a slope semistable system of
$\Omega_L$-Hodge sheaves.
\end{Corollary}

\subsection{Higgs-de Rham sequences}

Let $k$ be an algebraically closed field of characteristic $p>0$.
Let $X$ be a smooth projective $k$-variety of dimension $n$ that
can be lifted to a smooth scheme $\cX$ over $W_2(k)$.

Let $\MIC _{p-1}(X/k)$ be the category of $\cO_X$-modules with an
integrable connection whose $p$-curvature is nilpotent of level
less or equal to $p-1$. Similarly, let $\HIG _{p-1}(X/k)$ denote
the category of Higgs $\cO_{X'}$-modules with a nilpotent Higgs
sheaf of level less or equal to $p-1$. In this case one of the
main results of Ogus and Vologodsky (see \cite[Theorem 2.8]{OV})
says that:

\begin{Theorem} \label{Ogus-Vologodsky}
The Cartier operator
$$C_{\cX/\cS}:\MIC _{p-1}(X/k)\to \HIG _{p-1}(X'/k)$$
defines an equivalence of categories with quasi-inverse
$$C_{\cX/\cS}^{-1}:\HIG _{p-1}(X/k)\to \MIC _{p-1}(X'/k).$$
\end{Theorem}

A small variant of the following lemma can be found in proof of
\cite[Theorem 4.17]{OV}:

\begin{Lemma}\label{Chern-classes-of-Cartier-images}
Let $(E, \theta)\in \HIG _{p-1}(X'/S)$. Then
$$[C_{\cX/\cS}^{-1}(E)]=F_{X/S}^*[E],$$
where $[\cdot ] $ denotes the class of a coherent $\cO_X$-module
in Grothendieck's $K$-group $K_0(X)$.
\end{Lemma}

As a corollary to Theorem \ref{Ogus-Vologodsky} and Lemma
\ref{Chern-classes-of-Cartier-images} we get the following:

\begin{Corollary}\label{Cor-OV}
Let $(E, \theta)$ be a torsion free Higgs sheaf with nilpotent
Higgs field of level less than $p$. Then it is slope semistable if
and only if the corresponding sheaf with integrable connection
$(V, \nabla):=C_{\cX/\cS}^{-1}(E, \theta)$ is slope semistable.
\end{Corollary}

\medskip

Now let $(E, \theta )$ be a rank $r$ torsion free Higgs sheaf with
nilpotent Higgs field. Let us assume that $r\le p$ so that level
of nilpotence of $(E, \theta)$ is less than $p$. Let us recall the
following definition taken from \cite{LSZ}.

\begin{Definition}
A \emph{Higgs--de Rham sequence} of $(E, \theta)$ is an infinite
sequence
$$ \xymatrix{
& (V_0, \nabla _0)\ar[rd]^{\Gr _{N_0}}&& (V_1, \nabla _1)\ar[rd]^{\Gr _{N_1}}&\\
(E_0, \theta _0)=(E, \theta)\ar[ru]^{C^{-1}}&&(E_1, \theta_1)\ar[ru]^{C^{-1}}&&...\\
}$$ in which $C^{-1}=C^{-1}_{\cX/S}$ is the inverse Cartier
transform, $N_i^{\bullet}$ is a Griffiths transverse filtration of
$(V_i, \nabla _i)$ and $(E_{i+1}:=\Gr _{N_{i}}(V_{i}), \theta
_{i+1})$ is the associated Higgs sheaf.
\end{Definition}

\medskip

The following theorem proves the conjecture of Lan-Sheng-Zuo
\cite[Conjecture 2.8]{LSZ}:

\begin{Theorem}\label{LSZ_conjecture}
If $(E, \theta)$ is slope semistable then there exists a
canonically defined Higgs--de Rham sequence
$$ \xymatrix{
& (V_0, \nabla _0)\ar[rd]^{\Gr _{N_S}}&& (V_1, \nabla _1)\ar[rd]^{\Gr _{N_S}}&\\
(E_0, \theta _0)=(E, \theta)\ar[ru]^{C^{-1}}&&(E_1, \theta_1)\ar[ru]^{C^{-1}}&&...\\
}$$ in which each $(V_i, \nabla_i)$ is slope semistable
and $(E_{i+1}, \theta _{i+1})$ is the slope semistable
Higgs sheaf associated to $(V_i, \nabla _i)$ via Simpson's filtration.
\end{Theorem}

\begin{proof}
The proof is by induction on index $i$. Once we defined slope
semistable $(E_{i}, \theta_i)$, we can construct $(V_i, \nabla
_i)$, which is slope semistable by Corollary \ref{Cor-OV}. So by
Theorem \ref{existence-of-gr-ss-Griffiths-transverse-filtration}
there exists Simpson's filtration on $(V_i, \nabla _i)$ and hence
we can  construct a slope semistable Higgs sheaf $(E_{i+1},
\theta_{i+1})$. Since $(E_{i+1}, \theta_{i+1})$ is a system of
Hodge sheaves and $r\le p$, it satisfies the nilpotence condition
required to define $C^{-1}$.
\end{proof}

\medskip
In the above theorem slope semistability is defined with respect
to an arbitrary fixed collection $(D_1,\dots ,D_{n-1})$ of nef
divisors on $X$.

\bigskip

\emph{\large \bf Acknowledgements.}

The author would like to thank Carlos Simpson for explaining to
him  Lemma \ref{Simpson}. The author would also like to thank 
Yanhong Yang and the referee for pointing out a small problem with 
the first formulation of Corollary \ref{YY}.
Finally, the author would like to thank Kang Zuo for inviting him 
to the Johannes Gutenberg University Mainz to lecture on the material 
contained in the present paper and SFB/TRR 45 
``Periods, Moduli Spaces and Arithmetic of Algebraic Varieties'' 
for supporting his visit.

\end{document}